\documentclass{amsart}
\usepackage{amssymb,amsthm,amsmath}
\usepackage{mathtools}
\usepackage{bm}
\usepackage{esint}
\usepackage[hidelinks]{hyperref} 
\usepackage[noabbrev]{cleveref}
\usepackage{subcaption}
\usepackage[backend=biber,style=numeric,sorting=nyt,maxbibnames=99,
  minbibnames=99,url=false,
  doi=false,
  isbn=false,
  eprint=false]{biblatex}

\renewbibmacro*{volume+number+eid}{%
  \printfield{volume}%
  \setunit{\addspace}%
  \printtext{(\printdate)}%
  \iffieldundef{number}
    {}
    {\setunit{\addcomma\space}%
     \printtext{no.\addspace}\printfield{number}}%
  \setunit{\addcomma\space}%
}
\renewbibmacro*{issue+date}{} 

\DeclareFieldFormat[article]{pages}{#1}
  
\AtEveryBibitem{
  \clearfield{doi}
  \clearfield{isbn}
\clearfield{issn}
  \clearfield{eprint}
  \clearfield{language}
}
\newtheorem{theorem}{Theorem}[section]
\newtheorem*{conjecture*}{Conjecture}
\newtheorem{lemma}[theorem]{Lemma}
\newtheorem{proposition}[theorem]{Proposition}
\newtheorem{corollary}[theorem]{Corollary}
\theoremstyle{definition}
\newtheorem{definition}[theorem]{Definition}
\newtheorem{remark}[theorem]{Remark}
\newtheorem{example}[theorem]{Example}
\numberwithin{equation}{section}

\newcommand*{\de}{\partial}
\newcommand*{\N}{\mathbb{N}}
\newcommand*{\R}{\mathbb{R}}
\newcommand*{\NN}{\mathcal{N}}
\newcommand*{\sol}{\mathcal{WS}}

\newcommand*{\eps}{\varepsilon}

\newcommand*{\loc}{{\text{\upshape{loc}}}}
\newcommand*{\lap}{\Delta}
\newcommand*{\sing}{\Sigma}

\DeclareMathOperator{\osc}{osc}
\DeclareMathOperator{\avg}{avg}
\DeclareMathOperator{\dist}{dist}
\DeclareMathOperator{\reg}{Reg}

\usepackage{geometry}
\AddToHook{env/conjecture/begin}{\crefalias{theorem}{conjecture}}
\AddToHook{env/lemma/begin}{\crefalias{theorem}{lemma}}
\AddToHook{env/proposition/begin}{\crefalias{theorem}{proposition}}
\AddToHook{env/corollary/begin}{\crefalias{theorem}{corollary}}
\AddToHook{env/definition/begin}{\crefalias{theorem}{definition}}
\AddToHook{env/remark/begin}{\crefalias{theorem}{remark}}
\AddToHook{env/example/begin}{\crefalias{theorem}{example}}

\author{Alessio Figalli}
\address{ETH Z\"urich, R\"amistrasse 101, 8092 Z\"urich, Switzerland}
\email{alessio.figalli@math.ethz.ch}
\author{Federico Franceschini}
\address{Stanford University, 450 Jane Stanford Way, Stanford, CA 94305, USA}
\email{ffede@stanford.edu}
\date{\today}

\title[$\varepsilon$-regularity and Dimensional Bounds for the Singular Set]{Stable Semilinear Elliptic Equations:\\  $\varepsilon$-Regularity \`a la Brezis and\\ Dimensional Bounds for the Singular Set}

\begin{document}
\begin{abstract}
We develop a quantitative partial regularity theory for stable solutions of
\[
-\Delta u=f(u),
\]
where $f:\R \to [0,+\infty]$ is increasing and convex. The theory is uniform in the nonlinearity and
allows for a finite or infinite blow-up level $T_f\in(-\infty,+\infty].$


Our first result is a universal $\varepsilon$-regularity criterion that answers a celebrated question of Brezis: smallness of the scale-invariant mass of the stability potential $f'(u)$ forces H\"older regularity.
Moreover, if $T_f<+\infty$, the same smallness condition forces almost quadratic contact between the solution and the blow-up level $T_f$. This result is optimal and, in particular, covers the case of MEMS-type nonlinearities.

Our second result identifies a critical exponent \(q_f\ge1\), given explicitly in terms of the asymptotic behavior of \(f\), \(f'\), and \(f''\), such that
\[
f'(u)\in L^q_{\loc}\qquad\text{for every }q<q_f .
\]
Combined with our \(\varepsilon\)-regularity theorem, this yields quantitative bounds for the singular set, in particular
\[
\dim_{\mathcal H}\sing(u)\le n-2q_f.
\]
Remarkably, our exponent \(q_f\) recovers the sharp thresholds for all standard model nonlinearities, including \(f(t)=(1+t)^p\), \(e^t\), and \((1-t)^{-p}\). Also, this result provides the first general quantitative singular-set estimates for stable semilinear equations beyond the model nonlinearities.

Finally, in the two-dimensional case, we provide a complete picture by proving the universal Hessian estimate
\[
\|D^2u\|_{L^\infty(B_{1/2})}\le C\|u\|_{L^1(B_1)},
\]
where $C$ depends neither on $u$ nor on $f$. This \(C^{1,1}\) regularity is essentially optimal: one cannot expect \(C^{2,\alpha}\) estimates for any $\alpha>0$, and in general even \(C^2\) regularity should fail.
\end{abstract}

\maketitle

\section{Introduction}

Stable solutions of semilinear elliptic equations of the form
\[
-\Delta u=f(u)
\]
occupy a central place in nonlinear elliptic theory and arise naturally in a variety of problems. For specific nonlinearities, such as $e^u$, $(1+u)^p$, or $(1-u)^{-p}$, a well-developed theory identifies sharp regularity thresholds and, in several cases, provides quantitative information on the singular set.

The aim of this paper is to develop a sharp partial regularity theory for general positive, increasing, and convex nonlinearities, with estimates that are as uniform as possible with respect to $f$. This requires overcoming two structural difficulties. First, for an arbitrary nonlinearity $f$, there is no known monotonicity formula suitable for dimension reduction. Second, natural compactness arguments force one to consider limits in which both the solution and the nonlinearity vary, and the limiting nonlinearity may blow up at a finite level.

The main ingredient is a uniform $\varepsilon$-regularity theorem for stable solutions, which answers the intrinsic scale-invariant formulation \eqref{eq:Qbrezisinteg} of Brezis' question \eqref{eq:Qbrezis}; see \Cref{thm:epsreg} below. Before stating our results, we recall the classical Gelfand problem\footnote{Despite the name, this class of problems was first introduced by Barenblatt in a volume edited by Gelfand \cite{Gelfand1963Quasilinear}.}.

\subsection{Motivation: the Gelfand problem and a question of Brezis}

We begin with the classical Gelfand-type Dirichlet problem: for a parameter $\lambda>0$, consider the PDE
\begin{equation}
\label{eq:lambda}
\left\{
\begin{array}{cl}
-\Delta u=\lambda f(u) & \text{in }\Omega,\\
u>0 & \text{in }\Omega,\\
u=0 & \text{on }\partial\Omega,
\end{array}
\right.
\end{equation}
where $\Omega\subset \R^n$ is a bounded smooth domain and the nonlinearity $f\in C^2(\R)$ satisfies
\[
f>0,\qquad  f'\ge 0,\qquad  f''\ge 0 \qquad\text{and}\qquad  \frac{f(t)}{t}\to+\infty\ \text{ as }t\to+\infty.
\]
Let us establish some terminology for solutions of \eqref{eq:lambda}:
\begin{itemize}
    \item $u\in C^2(\Omega)\cap C(\overline \Omega)$ is a \textit{classical solution} if it solves \eqref{eq:lambda}.
    \item $u \in H^1_0(\Omega)$ is a \textit{weak solution} if $f(u)\cdot\dist(\cdot,\partial \Omega)\in L^1(\Omega)$ and
    \[
    \int_\Omega \nabla u\cdot \nabla \zeta\,dx=\lambda\int_{\Omega}f(u)\,\zeta\,dx
    \qquad\text{for all }\zeta \in C^{1}_c(\Omega).
    \]
    \item $u \in L^1(\Omega)$ is an $L^1$\textit{-weak solution} if $f(u)\cdot\dist(\cdot,\partial \Omega)\in L^1(\Omega)$ and
    \[
    -\int_\Omega u\, \Delta \xi\,dx=\lambda\int_{\Omega}f(u)\,\xi\,dx
    \qquad\text{for all }\xi \in C^{2}(\overline\Omega)\ \text{with } \xi|_{\partial\Omega}= 0.
    \]
\end{itemize}
A solution (of any of the above types) is said to be \textit{stable} if the linearized operator is nonnegative in the Dirichlet sense, i.e.
\[
\lambda_1(-\lap -\lambda f'(u),\Omega)\ge 0.
\]
Equivalently (and this is the form we will use throughout), $u$ is stable if  $f'(u)\in L^1_\loc(\Omega)$ and
\begin{equation}\label{eq:stabilityintro}
\int_\Omega\lambda \, f'(u) \,\xi^2\,dx \le \int_\Omega|\nabla\xi|^2\,dx
\qquad \text{for all }\xi\in C^1_c(\Omega).
\end{equation}

\smallskip

The main qualitative results for \eqref{eq:lambda} can be summarized as follows:
\begin{theorem}[\cite{brezisfailure,actadim9,martel}]\label{thm:lambda-star}
There exists a constant $\lambda^\star \in (0,+\infty)$ such that:
\begin{itemize}
\item[(i)] For every $\lambda \in (0,\lambda^\star)$ there is a unique weak stable solution $u_\lambda$, which is also classical. Furthermore, $u_\lambda>u_{\lambda'}$ for $\lambda>\lambda'$.
\item[(ii)] For every $\lambda >\lambda^\star$ there is no $L^1$-weak solution.
\item[(iii)] For $\lambda=\lambda^\star$ there exists a unique $L^1$-weak solution $u^\star$, which is also a weak stable solution and $u_\lambda\uparrow u^\star$ as $\lambda\uparrow \lambda^\star$. Additionally, if $n\le 9$, then $u^\star$ is necessarily classical.
\end{itemize}
\end{theorem}

\begin{remark}
A convenient way to construct $u_\lambda$ is via monotone iteration: set $u_{0,\lambda}=0$ and solve
\[
-\lap u_{k+1,\lambda}=\lambda f(u_{k,\lambda})\quad\text{in }\Omega,\qquad u_{k+1,\lambda}=0\quad\text{on }\partial\Omega.
\]
Then for $\lambda$ small one proves that the solutions remain bounded in $L^1(\Omega)$ and $u_{k,\lambda}\uparrow u_{\lambda}$, while for $\lambda$ large there is complete blow-up, in the sense that $\{u_{k,\lambda}(x)\}_k$ diverges for each $x\in\Omega$.
\end{remark}

\smallskip

In the influential work \cite{brezisfailure}, Brezis asked several questions concerning \eqref{eq:lambda}, especially regarding the properties of the extremal solution $u^\star$.
The starting point of the present work is a resolution of a natural scale-invariant formulation of
\cite[Open problem~4]{brezisfailure}:

\begin{center}
{\it ``Suppose $u^\star$ has an isolated singularity at $x_0\in\Omega$. Is it true that
\begin{equation}\label{eq:Qbrezis}
    f'(u^\star(x))\simeq \frac{1}{|x-x_0|^2}\quad\text{as }x\to x_0\text{?''}
\end{equation}}
\end{center}
This conjecture is motivated by two considerations. First, for the model nonlinearities
\[
f(u)=e^u \quad\text{or}\quad f(u)=(1+u)^p,
\]
one can construct explicit radial extremal solutions \(u^\star\) for which the conjecture holds. Second, the scaling \(1/|x|^2\) is the right one when comparing the stability condition \eqref{eq:stabilityintro} with Hardy's inequality, which states that for \(n\ge 3\),
\begin{equation}\label{eq:hardy}
\frac{(n-2)^2}{4}\int_{B_1} \frac{\xi(x)^2}{|x|^2}\,dx \le \int_{B_1} |\nabla \xi(x)|^2\,dx
\qquad \forall\,\xi\in C^1_c(B_1).
\end{equation}
When \(n=2\), Hardy's inequality includes a logarithmic correction:
\begin{equation}\label{eq:2Dhardy}
    \frac14\int_{B_1}\frac{\xi(x)^2}{|x|^2 (1-\log|x|)^2} \,dx \le \int_{B_1} |\nabla \xi(x)|^2\, dx\qquad\forall\,\xi\in C^1_c(B_1).
\end{equation}

\smallskip

If one interprets \eqref{eq:Qbrezis} in a pointwise sense, then Brezis' question has been answered in the negative by Villegas \cite{villegas21}: the quantity $f'(u^\star)$ can oscillate.
More precisely, Villegas managed to construct a convex, increasing nonlinearity $f_0\in C^\infty(\R)$ and a corresponding radial weak stable solution $u^\star_0$ such that
\[
0=\liminf_{r\to 0} r^2f'_0(u^\star_0(r))
<\limsup_{r\to 0} r^2f'_0(u^\star_0(r))=\frac{(n-2)^2}{4}.
\]
This suggests that Brezis' question may only be true in a limsup sense; indeed, still in \cite{villegas21}, Villegas proved the validity of the corresponding limsup statement in the radial case. Our \Cref{thm:sellpoint} will, in particular, extend his result to the general case.

\smallskip

One should note that the pointwise form \eqref{eq:Qbrezis} is meaningful only for isolated singularities and, if interpreted with universal constants at definite scales, requires quantitative separation from the rest of the singular set. An intrinsic formulation, which remains meaningful for arbitrary singular sets and is not merely a technical weakening, is instead in terms of the scale-invariant mass of the stability potential at the singular point:
\begin{equation}\label{eq:Qbrezisinteg}
   \text{{\it Let $u^\star$ be singular at $x_0$. Is $\limsup_{r\to 0}r^2\fint_{B_r(x_0)}\lambda^\star f'(u^\star)$ bounded below by a positive dimensional constant?}}
\end{equation}

\begin{remark}\label{rmk:L1bound}
The scale-invariant quantity
\[
r^2\fint_{B_r(x_0)}\lambda^\star f'(u^\star)
\]
is always bounded above by a dimensional constant (at least away from $\partial\Omega$).
Indeed, fix $r>0$ such that $B_{2r}(x_0)\subset\Omega$. Testing the stability inequality \eqref{eq:stabilityintro} with a cutoff function $\xi$ such that
$\xi=1$ on $B_r(x_0)$, $\xi=0$ on $\Omega\setminus B_{2r}(x_0)$, and $|\nabla \xi|\leq C(n)r^{-1}$, yields
\[
\int_{B_r(x_0)}\lambda^\star f'(u^\star)\,dx
\le \int_{\Omega}\lambda^\star f'(u^\star)\,\xi^2\,dx
\le \int_{\Omega}|\nabla\xi|^2\,dx
\le C(n)r^{n-2}.
\]
Equivalently, $r^2\fint_{B_r(x_0)}\lambda^\star f'(u^\star)\le C(n)$.
Thus \eqref{eq:Qbrezisinteg} can be viewed as asking whether the \emph{smallness} of this scale-invariant quantity forces regularity near $x_0$.
\end{remark}

\subsection{First main result: \texorpdfstring{$\eps$}{}-regularity and the Brezis question}

Our first main result gives a complete answer to the scale-invariant question \eqref{eq:Qbrezisinteg}, proving that at any singular point the scale-invariant mass of $f'(u^\star)$ cannot be arbitrarily small along all scales. As discussed in 
Remark~\ref{rmk:noliminf}, this result is optimal.

\begin{theorem}\label{thm:sellpoint}
Let $u^\star$ be the extremal solution of \eqref{eq:lambda} at $\lambda=\lambda^\star$, and let $x_0\in\Omega$. There is a dimensional constant $C=C(n)>0$ such that either $u^\star$ is classical in a neighborhood of $x_0$, or
\[
\frac1C \le \limsup_{r\to 0} r^{2-n}\int_{B_r(x_0)}\lambda^\star f'(u^\star)\,dx \le C.
\]
\end{theorem}

\Cref{thm:sellpoint} is a consequence of the more general \Cref{thm:epsreg} and \Cref{cor:introshape} below.
\begin{remark}\label{rmk:pointwise}
    In the case of an isolated singularity $x_0\in\Omega$, using the universal upper bound to pass from balls to dyadic annuli, \Cref{thm:sellpoint} gives in particular a sequence $x_k\to x_0$ such that $\lambda^\star f'(u^\star(x_k))\ge C(n) |x_k-x_0|^{-2}$. This provides the limsup lower-bound conclusion in the original pointwise formulation of Brezis' question \eqref{eq:Qbrezis}.
\end{remark}
\begin{remark}\label{rmk:noliminf}
A modification of the examples of Villegas (see Appendix~\ref{app:Villegas}) shows that the ``limsup'' in \Cref{thm:sellpoint} cannot be improved to a ``liminf''.
Indeed, we can construct an explicit nonlinearity $f_1\in C^\infty(\R)$ (convex, increasing, positive) and a corresponding radial \textit{singular} weak stable solution $u_1^\star$ such that
\[
\liminf_{r\to 0} r^{2-n}\int_{B_r(x_0)}\lambda^\star f'_1(u^\star_1)=0,
\]
so \Cref{thm:sellpoint} is optimal.
\end{remark}

\medskip

\Cref{thm:sellpoint} is only the starting point of the present paper, which contains a number of new results.
By way of a summary, the following are the new contributions of this work as far as $\eps$-regularity is concerned:
\begin{itemize}
\item The lower bound in \Cref{thm:sellpoint} can be reformulated as a universal $\eps$-regularity theorem; see \Cref{thm:epsreg}.
This was previously known only for specific model nonlinearities, while here $\eps$ is \emph{purely dimensional} and, perhaps surprisingly, does not depend on $f$.
\item We work in the closed class of PDEs obtained as limits of stable solutions of $-\lap u_k=f_k(u_k)$ with $f_k$ positive, increasing, and convex (possibly taking the value $+\infty$).
This class, introduced in \cite{actadim9} and denoted here by $\sol$, is closed (note that both $u$ and $f$ vary) and has natural notions of regular and singular sets.
We prove that the regular set is \emph{always} open; see \Cref{thm:open}.  
\item In \Cref{thm:C1alpha} we identify the sharp obstruction to $\eps$-regularity in the finite-blow-up case, and in \Cref{cor:assmems} we show that MEMS-type growth rules it out.
\end{itemize}

\subsection{Second main result: a dimensional bound on the singular set}

With $\eps$-regularity at hand, one can hope to estimate the size of the singular set for general nonlinearities.
Note that, in this context, no ``dimension reduction principle'' is available due to the lack of a suitable monotonicity formula.
Hence, knowing that no singular solutions exist below a certain dimension does \emph{not} directly imply a corresponding codimension bound on $\sing(u^\star)$.

\smallskip

Our strategy is the following: we show that higher integrability of $f'(u^\star)$ implies quantitative bounds on the size of the singular set.
More precisely, as a direct application of our $\eps$-regularity theorems, we prove that $L^q$ bounds on $f'(u^\star)$ yield $(n-2q)$-dimensional bounds on $\sing(u^\star)$.
Thus, in order to bound $\sing(u^\star)$, it suffices to establish higher integrability estimates on $f'(u^\star)$.\footnote{The idea of proving higher integrability estimates on $f'(u^\star)$ in the stable regime is classical; see for instance \cite{dupaigne} and references therein. However, most available results only concern specific classes of nonlinearities.}

Let $T_f \in (-\infty,+\infty]$ denote the blow-up level of $f$ (see \eqref{eq:defT} below).
In \Cref{thm:dimbound}, under mild asymptotic assumptions on $f$, we prove that
$f'(u^\star) \in L^q_{\loc}$ for every $q<q_f$, where $q_f$ is defined by
\begin{equation}\label{eq:q_f1}
q_f:=1+ 2\liminf_{t\uparrow T_f}\frac{\log f(t) + \int^t_{} \sqrt{\frac{f''(s)}{f(s)}}\,ds}{\log f'(t)}
\end{equation}
(since $\log f'(t) \to +\infty$ as $t\uparrow T_f$, the choice of the lower limit in the integral does not affect the value of $q_f$).
As a consequence, we obtain
\begin{equation}\label{eq:q_f2}
\dim_{\mathcal{H}}\sing(u^\star) \leq n-2q_f.
\end{equation}
Although the expression \eqref{eq:q_f1} may appear somewhat unusual at first glance,
the quantity $q_f$ arises naturally from the structure of the stability inequality underlying our estimates.
Remarkably, it also recovers the optimal exponent in all standard model nonlinearities, for instance for
\[
f(t)=(1+t)^p\ (p>1),\qquad f(t)=e^t,\qquad f(t)=(1-t)_+^{-p}\ (p>0),\quad \ldots
\]
(see \Cref{rem:model}, Section~\ref{sect:compare}, and also \Cref{rmk:liminfqf}).
This agreement with the known sharp examples indicates that $q_f$ is the natural quantity governing the singular set in general,
and leads to the following optimality conjecture.

\begin{conjecture*}
Let $q_f$ be defined by \eqref{eq:q_f1}. Within the smooth asymptotic class of \Cref{thm:dimbound}, the bound \eqref{eq:q_f2} cannot in general be improved.
\end{conjecture*}

We note that, at this level of generality, quantitative bounds on the singular set were not previously available:
earlier results either proved full regularity in low dimensions or obtained bounds on the singular set only for very specific model nonlinearities.

\subsection{Third main result: the two-dimensional case}
In dimension two, \Cref{cor:2Dmems} below shows that the singular set is empty whenever our $\varepsilon$-regularity theorem applies. In particular, if $T_f<+\infty$, this follows under a MEMS-type growth assumption, which is sharp in dimensions $n\ge3$ (see \Cref{ex:r2logr}). We can actually prove a much stronger conclusion: in the full local class $\sol(B_1)$ introduced in \eqref{def:Somega}, every local weak stable solution satisfies a universal interior Hessian estimate (previously known, by a different argument, in the radial case).

\begin{theorem}\label{thm:2Dintro}
Let $(u,f)\in\sol(B_1)$, with $B_1\subset\R^2$. Then
\[
\|D^2u\|_{L^\infty(B_{1/2})}\le C\,\|u\|_{L^1(B_1)},
\]
where \(C\) is universal.
In particular, if $\sing(u)\neq\emptyset$, then necessarily
\[
T_f<+\infty
\qquad\text{and}\qquad
f(T_f)<+\infty.
\]
\end{theorem}

Note that this theorem does not imply that the singular set is empty: indeed, \(T_f\) may be finite with $f(T_f)<+\infty$, while \(f'(T_f)=+\infty\). As discussed in \Cref{rmk:finitef} and \Cref{rmk:weiss-monneau}, this $C^{1,1}_{\loc}$ regularity is expected to be optimal.

The key new ingredient is a universal estimate for \(f(u)\), proved by reducing suitable averages of \(f(u)\) to a delayed ODE. 
Once \(f(u)\) is universally bounded, an intrinsic rescaling argument of obstacle-problem type near high-level contact points yields an \(L^\infty\)-based Hessian estimate, and a final interpolation step gives the \(L^1\) bound above.
\bigskip

In the next section, we state our results precisely and extend the discussion above.

\section{Statement of the results}

\subsection{Local weak stable solutions}

In order to study blow-up limits of stable solutions (where both the solution and the nonlinearity may vary),
we work in the closed class introduced in \cite{actadim9}.
Throughout, we allow nonlinearities that may take the value $+\infty$ and may blow up at a finite level.

\smallskip

Consider the class of nonlinearities
\begin{equation}\label{eq:defofC}
    \NN := \big\{ f\colon \R \to [0,+\infty] : f \text{ is lower semicontinuous, nondecreasing, and convex}\big\}.
\end{equation}
For $f\in\NN$ we denote by $f'$ the \emph{left derivative}, namely
\begin{equation*}
f'(t):=\sup_{\delta>0}\frac{f(t)-f(t-\delta)}{\delta}.
\end{equation*}
By convexity, $f'$ is defined everywhere on $\R$, is nondecreasing, and is left-continuous.

\medskip

We also introduce the (possibly infinite) blow-up level of $f$:
\begin{equation}\label{eq:defT}
T_f:=\inf\{ s\in\R : f'(s)=+\infty \}\in(-\infty,+\infty].
\end{equation}
When $T_f<+\infty$, we will refer to $f$ as a \emph{singular} nonlinearity; otherwise ($T_f=+\infty$) we call it \emph{regular}.
We use the convention
\[
f(T_f):=\lim_{t\uparrow T_f}f(t)\in[0,+\infty],
\]
including the case \(T_f=+\infty\).

\begin{definition}\label{def:weakstablesol}
    Let $f\in\NN$ and $\Omega\subset\R^n$ be open.
    We say that $u\in H^1_\loc(\Omega)$ is a \emph{local weak stable solution} of
\begin{equation}\label{eq:PDE}
-\Delta u = f(u)
\end{equation}
if $f(u),\,f'(u)\in L^1_\loc(\Omega)$ and, for all $\xi \in C^1_c(\Omega)$, it holds
\begin{equation}\label{eq:stability}
	\int_\Omega \nabla u\cdot\nabla\xi\,dx = \int_\Omega f(u)\,\xi\,dx
	\qquad\text{and}\qquad
	\int_\Omega |\nabla\xi|^2\,dx \ge \int_\Omega f'(u)\,\xi^2\,dx .
\end{equation}
\end{definition}

We denote the class of such pairs by
\begin{equation}\label{def:Somega}
    \sol(\Omega):=\big\{(u,f): \text{$u$ is a local weak stable solution of \eqref{eq:PDE} in }\Omega\big\}
    \subset H^1_\loc(\Omega)\times \NN .
\end{equation}

\begin{remark}
In the framework of the Gelfand problem \eqref{eq:lambda} (with the notations of \Cref{thm:lambda-star}),
one has $(u_\lambda,\lambda f)\in\sol(\Omega)$ for all $\lambda\le \lambda^\star$.
\end{remark}

\medskip

Since $f\ge 0$, equation \eqref{eq:PDE} implies $\Delta u\le 0$ in the weak sense, hence $u$ is superharmonic.
In particular, $u$ admits a canonical pointwise representative (possibly taking the value $+\infty$ somewhere) defined at every $x\in\Omega$ by
\[
u(x)=\lim_{r\downarrow 0}\fint_{B_r(x)} u\,dx
      =\sup_{r>0}\fint_{B_r(x)} u\,dx
\quad\in\quad (-\infty,T_f],
\]
and this representative is lower semicontinuous. We will always work with this pointwise-defined representative.

\medskip

There is a natural notion of regular and singular sets (relative to the blow-up level $T_f$).

\begin{definition}
We define
\[
\reg(u):=\{x\in\Omega:\ u(x)<T_f\},
\qquad
\sing(u):=\{x\in\Omega:\ u(x)=T_f\}.
\]
\end{definition}

Our first result in this abstract framework shows that, despite $u$ being \emph{a priori} only lower semicontinuous,\footnote{\,If $u$ were upper semicontinuous, then the openness of $\reg(u)$ would be obvious.}
the set $\reg(u)$ is always open and $u$ is continuous (with the correct interpretation near $\sing(u)$).

\begin{theorem}\label{thm:open}
Let $(u,f)\in \sol(\Omega)$. Then $\reg(u)$ is open, the map
\[
u\colon \Omega \to (-\infty,T_f]
\]
is continuous as an extended-real-valued map,\footnote{\,When $T_f=+\infty$, continuity means in particular that $u(x)\to+\infty$ as $\dist(x,\sing(u))\to 0$.}
and $u$ is of class $C^{2,\alpha}$ in $\reg(u)$ for all $\alpha<1$.
\end{theorem}

\begin{remark}\label{rmk:salt}
When $T_f<+\infty$, one can design $f$ in such a way that $f'(u(x_0))=+\infty$ while $f(u(x_0))<+\infty$, an explicit example is provided in \Cref{rmk:finitef} below along with a more detailed discussion. We decided to include these points in the singular set.
\end{remark}

\medskip

The class $\sol$ enjoys a strong closedness property, which will be used repeatedly in the proofs of our $\eps$-regularity results.
This is essentially contained in \cite{actadim9}; however, since the proof of \eqref{eq:specification} is only implicit there,
we fill the details in Appendix~\ref{app:ACTA}.

\begin{theorem}[{\cite[Theorem 4.1]{actadim9}}]\label{thm:compactness}
Let $\{(u_k,f_k)\}_{k\in\N}\subset\sol(\Omega)$ and assume
\[
u_k\to u\quad\text{in }L^1_\loc(\Omega).
\]
Then there exists $f\in \NN$ such that $(u,f)\in \sol(\Omega)$, $u_k\to u$ strongly in $W^{1,2}_\loc(\Omega)$, and
\[
f_k\to f \quad\text{locally uniformly on compact subsets of }(-\infty,\sup_{\Omega}u).
\]
Furthermore, if $u$ is not constant, 
\begin{equation}\label{eq:specification}
    f_k(u_k(x)) \to f(u(x))\quad\text{ as }k\to+\infty,
\end{equation}
for almost every $x\in \Omega$.
\end{theorem}

\medskip

In the next subsections we state our $\eps$-regularity results.
We will distinguish between regular nonlinearities ($T_f=+\infty$) and singular nonlinearities ($T_f<+\infty$). We will address separately the two-dimensional case in \Cref{sec:2D}.

\subsection{\texorpdfstring{$\eps$}{}-regularity \`a la Brezis for regular nonlinearities}
Here, the expression ``\`a la Brezis'' refers to the connection with Brezis' question \eqref{eq:Qbrezis}:
the scale-invariant size of $f'(u)$ at a point controls whether that point can be singular.

\begin{theorem}\label{thm:epsreg}
There exists $\eps=\eps(n)>0$ such that the following holds.
Assume $(u,f)\in\sol(B_1)$, $T_f=+\infty$, and let $x_0\in B_{1/2}$.
If
\begin{equation}\label{eq:limsup}
\limsup_{r\downarrow 0} r^{2-n}\int_{B_r(x_0)} f'(u)\,dx <\eps,
\end{equation}
then $x_0\in \reg(u)$.
\end{theorem}

\begin{remark}\label{rmk:villegas}
In many $\eps$-regularity results one hopes to propagate smallness across scales,
implying that the ``limsup'' in assumption \eqref{eq:limsup} is replaced by a ``liminf''.
Here this is just \emph{not} possible: a modification of Villegas' counterexamples shows that in fact $f'(u)$ may oscillate strongly near a singularity, both in the pointwise and in the integral sense. See \Cref{rmk:noliminf} and Appendix~\ref{app:Villegas}.
\end{remark}

Combining \Cref{thm:epsreg} with the universal upper bound obtained by testing stability with a cutoff
(the analogue of \Cref{rmk:L1bound} in the local setting), we obtain:

\begin{corollary}\label{cor:introshape}
Let $(u,f)\in\sol(B_1)$, $T_f=+\infty$, and $x_0\in B_{1/2}\cap\sing(u)$.
Then
\[
\eps(n)\le \limsup_{r\downarrow 0} r^{2-n}\int_{B_r(x_0)} f'(u)\,dx \le C(n).
\]
\end{corollary}
\Cref{cor:introshape} and \Cref{rmk:villegas} give a complete answer to the natural limsup versions of Brezis' question \eqref{eq:Qbrezisinteg}, together with counterexamples to the corresponding liminf strengthening.
\subsection{\texorpdfstring{$\eps$}{}-regularity \`a la Brezis for singular nonlinearities}

When $T_f<+\infty$, the smallness of the scale-invariant quantity $r^{2-n}\int_{B_r(x_0)} f'(u)$ is not, by itself, enough to force $u(x_0)<T_f$.
This is because, in this case, at singular points the function $u$ may remain bounded (and even be $C^1$),
so detecting singular behavior requires finer information.

\begin{theorem}\label{thm:C1alpha}
Given any $\sigma>0$, there exists $\eps=\eps(n,\sigma)>0$ such that the following holds.
Assume $(u,f)\in\sol(B_1)$, $T_f<+\infty$, and $x_0\in B_{1/2}$.
If
\[
\limsup_{r\downarrow 0} r^{2-n}\int_{B_r(x_0)} f'(u)\,dx <\eps,
\]
then either $x_0\in\reg(u)$ or
\begin{equation}\label{eq:alternative}
\limsup_{x\to x_0}\frac{T_f-u(x)}{|x-x_0|^{2-\sigma}}<+\infty.
\end{equation}
\end{theorem}

Alternative \eqref{eq:alternative} means that the graph of $u$ in $B_1\times \R$ touches the hyperplane $\{x_{n+1}=T_f\}$ in a $C^{2-\sigma}$ fashion at $x_0$.
When $f$ blows up at least like some inverse power (MEMES-type behavior), this is incompatible with stability,
and we recover an analogue of \Cref{thm:epsreg}:

\begin{corollary}\label{cor:assmems}
Let $T_f<+\infty$ and assume
\begin{equation}\label{eq:assMEMS}
\limsup_{t\uparrow T_f }\, (T_f-t)^{p}f(t)>0
\quad\text{ for some }p>0.
\end{equation}
There exists $\eps=\eps(n,p)>0$ such that, if $(u,f)\in\sol(B_1)$ and $x_0\in B_{1/2}$ satisfy
\[
\limsup_{r\downarrow 0} r^{2-n}\int_{B_r(x_0)} f'(u)\,dx <\eps,
\]
then $x_0\in\reg(u)$.
\end{corollary}

The following explicit example shows that assumption \eqref{eq:assMEMS} is necessary.

\begin{example}\label{ex:r2logr}
For $n\ge 3$ and $r_0>0$ small, consider the decreasing radial function
\[
u(r):= 1-\log(1/r)\, r^2,\qquad r\in(0,r_0].
\]
A direct computation gives
\[
\Delta u(r)=u''(r)+\frac{n-1}{r}u'(r)=n+2-2n\log(1/r).
\]
Since $u\colon(0,r_0]\to (u(r_0),1]$ is a $C^1$ diffeomorphism for $r_0$ small,
we can define implicitly a nonlinearity $f:(u(r_0),1]\to[0,+\infty)$ via the relation $f(u)=-\Delta u$, that in this case takes the form
\begin{equation}\label{eq:54}
f\bigl(1-r^2\log(1/r)\bigr)=2n\log(1/r)-n-2.
\end{equation}
One checks that $f$ is increasing, convex, and blows up at $1$.
Moreover, differentiating \eqref{eq:54} yields
\[
f'(u(r))=\frac{2n}{(2\log(1/r)-1)\,r^2},
\qquad\text{hence}\qquad
r^2 f'(u(r))=\frac{2n}{2\log(1/r)-1}\xrightarrow[r\downarrow 0]{}0.
\]
In particular, 
\[
(1-t)\,f'(t)=n+o(1)\qquad\text{as }t\uparrow 1,
\]
and $f$ blows up logarithmically as $t\uparrow 1$.
Consequently, for $r_0$ small enough the stability inequality holds (since $f'(u(r))=o(r^{-2})$ is dominated by the Hardy weight),
and $0$ is a singular point (since $u(0)=1=T_f$).
However, in this case,
\[
\lim_{r\downarrow 0}r^{2-n}\int_{B_r}f'(u)\,dx =0,
\]
so the conclusion of \Cref{cor:assmems} does not hold.
We can note that indeed \eqref{eq:alternative} holds for every $\sigma>0$ (since $r^\sigma\log(1/r)\to 0$).
\end{example}

It is not a coincidence that this example requires $n\ge 3$: in \Cref{sec:2D} below we show that $C^{1,1}$ regularity holds in two dimensions, meaning that one can take $\sigma=0$. In particular we prove that, without any further assumption, if $f(T_f)=+\infty$ then $\sing(u)=\emptyset$.

\subsection{Higher integrability of \texorpdfstring{${f'(u)}$}{} and the size of the singular set}\label{subsec:brezsize}

Since solutions of $-\Delta u=f(u)$ lack a monotonicity formula suitable for dimension reduction,
our approach to bounding $\sing(u)$ uses the following standard principle:
\begin{center}
{\it Higher integrability of $f'(u)$ forces quantitative smallness of the set of points\\ where the scale-invariant mass of $f'(u)$
stays bounded from below.}
\end{center}
Concretely,
\begin{equation}\label{eq:Lpbound}
f'(u)\in L^q_\loc,\quad q\ge 1
\quad\Longrightarrow\quad
\mathcal H^{n-2q}(\sing(u)) =0.
\end{equation}
This implication follows by combining our $\eps$-regularity statements (namely \Cref{thm:epsreg} and \Cref{cor:assmems})
with the following classical covering lemma (for the reader's convenience, we provide a proof in Appendix~\ref{app:3}).

\begin{lemma}\label{lem:gmtintro}
Let $\varphi \in L^q_\loc(\R^n)$ be nonnegative, $\alpha>0$, $q\ge 1$, and $E\subset B_1$ be closed.
Assume
\[
E\subset \bigg\{x\in B_1 : \limsup_{r\downarrow 0} r^{\alpha-n}\int_{B_r(x)} \varphi\,dx \ge 1 \bigg\}.
\]
Then
\[
\mathcal{H}^{n-\alpha q}(E)=0.
\]
Here, if \(n-\alpha q\le0\), the conclusion is understood as \(E=\emptyset\).
\end{lemma}

Since stability provides an a priori $L^1$ bound on $f'(u)$ in interior balls (simply test \eqref{eq:stability} with a cutoff),
we immediately deduce the following very general result.

\begin{corollary}\label{cor:2Dmems}
Let $(u,f)\in\sol(B_1)$ and assume either
\[
T_f=+\infty
\qquad\text{or}\qquad
T_f<+\infty\ \text{ and \eqref{eq:assMEMS} holds.}
\]
Then
\[
\mathcal{H}^{n-2}(\sing(u))=0.
\]
\end{corollary}

\begin{remark}
The dimensional bound $n-2$ is optimal as a uniform bound over the class $\sol$, as the example $f(u)=(1-u)^{-p}$ for $p>0$ small shows, see Example~\ref{rem:model}(iii) below. 
\end{remark}

\medskip

If we fix a certain $f\in\NN$, we can improve \Cref{cor:2Dmems} by proving higher integrability estimates
$f'(u)\in L^q$ for some $q>1$ depending on $f$.
For this, we need to impose two assumptions:

\begin{itemize}
\item[$\bullet$] There exists $\gamma_\circ>0$ such that
\begin{equation}\label{eq:loggrowth}
\liminf_{t\to +\infty} \frac{f'(t)}{(\log t)^{2+\gamma_\circ}}>0.
\end{equation}
(For $T_f<+\infty$ this condition is automatically satisfied.)

\item[$\bullet$] There exist $m_f>0$ and a continuous increasing function $\omega\colon[0,+\infty)\to[0,+\infty)$, with $\omega(0)=0$,
such that $f(m_f)>0$, $f'(m_f)>0$, and
\begin{equation}\label{eq:weakCR}
\frac{f'(s)}{f'(t)} \le \omega\bigg(\frac{f(s)}{f(t)} \bigg)
\qquad \text{for all } m_f \leq s<t<T_f.
\end{equation}
\end{itemize}

\begin{remark}\label{rmk:ass-vs-CR}
Assumption \eqref{eq:loggrowth} is a very mild growth hypothesis, satisfied in all model cases of interest.

Concerning \eqref{eq:weakCR}, since $f'(s)\le f'(t)$ for $s<t$,
this condition is only relevant in the regime where $f(s)/f(t)\ll 1$.
Based on this observation, we can show that \eqref{eq:weakCR} is robust under perturbations that preserve the growth of $f$. More precisely:
\begin{center}
{\it If $f$ satisfies \eqref{eq:weakCR} and $g\sim f$, then also $g$ satisfies \eqref{eq:weakCR} (possibly with a different $\omega$).}
\end{center}
The notation $g\sim f$ by definition means that for some constant $C=C(f,g)$ it holds
\begin{equation}\label{eq:fg-comparable}
\frac{1}{C}\,f(t)-C \leq g(t) \leq C\,f(t)+C\qquad \forall\,t \in \R,
\end{equation}
so in particular $T_g=T_f$. In order to check the claim we  consider two cases:\\
(i) If $\lim_{t\uparrow T_f}f(t)<+\infty$, then $f(s)/f(t)\sim 1$ for all $m_f \leq s<t<T_f$,
so \eqref{eq:weakCR} holds trivially.\\
(ii)
If instead $\lim_{t\uparrow T_f}f(t)=+\infty$, then \eqref{eq:fg-comparable} implies that
\[
\frac{f(s)}{f(t)} \leq C'\,\frac{g(s)}{g(t)}
\qquad\text{whenever }\frac{g(s)}{g(t)} \ll 1,
\]
and hence also $g$ satisfies a condition of the form \eqref{eq:weakCR} (with a different $\omega$).
\end{remark}

\begin{remark}
Assumption \eqref{eq:weakCR} is implied by the Crandall--Rabinowitz-type condition
\begin{equation}\label{eq:CRintro}
\liminf_{t\uparrow T_f}\frac{f''(t) f(t)}{f'(t)^2} >0,
\end{equation}
whenever $f\in C^2$.
Indeed, if \eqref{eq:CRintro} holds, then there exist $\delta,m_\delta>0$ such that
$f''(t) f(t)\ge \delta f'(t)^2$ for all $t\ge m_\delta$. This implies that $f'(t)\,f(t)^{-\delta}$ is increasing on $[m_\delta,T_f)$, therefore \eqref{eq:weakCR} holds with $\omega(s)=s^\delta$.

Because of this observation, one can immediately check that \eqref{eq:weakCR}
is satisfied by all classical examples of nonlinearities.
\end{remark}

\medskip

We now define the exponent $q_f\ge 1$ controlling the integrability of $f'(u)$ in the smooth asymptotic regime.
\begin{equation}\label{eq:q_f}
q_f:=1+ 2\liminf_{t\uparrow T_f}\frac{\log f(t) + \int^t_{m_f} \sqrt{{f''(s)}/{f(s)}}\,ds}{\log f'(t)}.
\end{equation}
Since $f'(t)\to+\infty$ as $t\uparrow T_f$, the precise choice of $m_f$ does not affect $q_f$.
Although this quantity may look unusual at first glance, Section~\ref{sect:compare} shows that our formula is sharp in all model cases, providing strong evidence that it is optimal in general.

\begin{theorem}\label{thm:dimbound}
Let $f\in \NN$ satisfy \eqref{eq:loggrowth} and \eqref{eq:weakCR}. Assume in addition that $f\in C^2((m_f,T_f))$ and that $f(t)\to+\infty$ as $t\uparrow T_f$.
Then for all $(u,f)\in\sol(\Omega)$ and all $q<q_f$,
\[
f'(u)\in L^{q}_\loc(\Omega).
\]
In particular,
\[
\dim_{\mathcal{H}}\sing(u)\le n-2q_f,
\]
provided
\[
T_f=+\infty
\qquad\text{or}\qquad
T_f<+\infty\ \text{and \eqref{eq:assMEMS} holds.}
\]
\end{theorem}

Combining \Cref{thm:dimbound} with Lemma~\ref{lem:logratio}, we get a general codimension bound.

\begin{corollary}\label{cor:dimbound}
Let \(\Omega\subset\R^n\) be open and let \((u,f)\in\sol(\Omega)\), where $f\in \NN$ satisfies the hypotheses of \Cref{thm:dimbound}. Assume in addition that
\[
\lim_{t\uparrow T_f}\frac{\log f(t)}{\log f'(t)}
\quad \text{exists.}
\]
\begin{enumerate}
\item If $T_f=+\infty$, then $\dim_{\mathcal{H}}\sing(u)\le n-6$.
\item If $T_f<+\infty$ and both $\int_0^{T_f}f(s)\,ds=+\infty$ and  \eqref{eq:assMEMS} hold, then $\dim_{\mathcal{H}}\sing(u)\le n-4$.
\end{enumerate}
Here, if the right-hand side is negative, the conclusion means \(\sing(u)=\emptyset\).
\end{corollary}

Finally, under a Crandall--Rabinowitz-type regime, one can obtain an explicit lower bound on $q_f$.

\begin{corollary}\label{cor:CRvsGamma}
Let $f\in \NN$ satisfy the hypotheses of \Cref{thm:dimbound}, and assume in addition that $f\in C^2((-\infty,T_f))$. Define $\gamma_\pm \in [0,+\infty]$ as
\[
\gamma_-:= \liminf_{t\uparrow T_f}\frac{f''(t) f(t)}{f'(t)^2},
\qquad
\gamma_+:= \limsup_{t\uparrow T_f}\frac{f''(t) f(t)}{f'(t)^2}.
\]
If $0<\gamma_+<+\infty$, then
\begin{equation}\label{eq:correctq}
q_f \ge 1+ \frac{2}{\gamma_+} +\frac{2\sqrt{\gamma_-}}{\gamma_+}.
\end{equation}
\end{corollary}

\begin{remark}\label{rmk:liminfqf}
The exponent $q_f$ in \eqref{eq:q_f} is defined via a $\liminf$, and we do not expect that one can
replace it by a $\limsup$ in the generality of \Cref{thm:dimbound}.
The reason is the following: Proposition~\ref{prop:aprioribound} provides a priori $L^1$ bound on the function $f'(u)\,f(u)^2\,H(u)$ where $H$ is a function depending on $f$ (see Definition~\ref{def:Hnu}), and then the $L^q$ estimate ultimately relies on a \emph{uniform} pointwise comparison for large $t$,
schematically of the form (cf. \eqref{eq:higherintegrability})
\[
f'(t)^q \;\lesssim\; f'(t)\,f(t)^2\,H(t)\qquad\text{for all $t$ sufficiently close to $T_f$},
\]
which is exactly the information captured by a $\liminf$ in \eqref{eq:q_f}.
A $\limsup$ would only control the right-hand side along a subsequence of $t$'s and would allow arbitrarily deep
downward oscillations at intermediate values, where the above comparison fails.

This ``worst--case'' nature is consistent with the oscillatory phenomena discovered by Villegas (see \cite{villegas21} and Appendix~\ref{app:Villegas} below):
in dimensions where singular stable solutions exist (e.g., $n\ge 10$), Villegas constructed convex increasing nonlinearities and
singular stable radial solutions for which the stability potential oscillates sharply across scales, in the sense that
$r^2 f'(u(r))$ can have $\liminf_{r\downarrow 0} r^2 f'(u(r))=0$ while $\limsup_{r\downarrow 0} r^2 f'(u(r))$ remains at the
Hardy threshold $(n-2)^2/4$.
Such examples show that, near a singularity, quantities governed by stability may be arbitrarily small along infinitely many
scales even if they are large along other scales. For general $f$ we therefore expect that the admissible integrability range
for $f'(u)$ is dictated by the least favorable asymptotic regime of $f$, and hence that the $\liminf$ in \eqref{eq:q_f} is the
natural (and in this sense optimal) quantity.
\end{remark}

\begin{remark}
If $f$ grows very slowly (e.g.\ $f(t)\sim t\log t$ as $t\to+\infty$), then \eqref{eq:loggrowth} fails and our result does not apply.
On the other hand, in standard slowly growing regimes the problem is subcritical in all dimensions and $\sing(u)=\emptyset$ follows from standard elliptic regularity and Sobolev embeddings.
\end{remark}

\begin{remark}\label{rmk:f'bound}
    We record that, in \Cref{prop:ff'} below, we additionally obtain  for all $(u,f)\in\sol(B_1)$ the universal local bound:
\begin{equation*}
 \int_{B_{1/2}} f'(u)f(u)\,dx
 \le C(n)\, \|u\|_{L^\infty(B_1)},
\end{equation*}
which was known only qualitatively (see for example \cite[Theorem 1]{DongFeng2}). 
\end{remark}

\subsection{Application of our results to model cases}\label{sect:model}

A useful quantity to discuss explicit examples with $f\in C^2$ is the function
\begin{equation}\label{eq:deftauintro}
\tau_f(t) := \frac{f''(t)\,f(t)}{f'(t)^2}.
\end{equation}

If $T_f=+\infty$ and $\tau_\infty:=\lim_{t\to+\infty}\tau_f(t)$ exists, then necessarily $\tau_\infty\le 1$
(otherwise $f$ would blow up in finite time).
Applying \eqref{eq:correctq} with $\gamma_-=\gamma_+=\tau_\infty$ yields $q_f\ge 5$ and hence $\dim_{\mathcal{H}}\sing(u)\le n-10$, the expected bound in view of \cite{actadim9}.

Let us discuss some models examples.

\begin{example}\label{rem:model}
Let $(u,f)\in\sol(B_1)$ with $B_1\subset\R^n$.
\begin{itemize}
\item[(i)] If $f(t)=(1+t)^p$ with $p>1$, then $\tau_f(t)\equiv (p-1)/p$, so
\[
q_f=1+2\sqrt{\frac{p}{p-1}} +2\frac{p}{p-1},
\qquad
\dim_{\mathcal{H}} \sing(u)\le n-2-4\frac{p}{p-1}-4\sqrt{\frac{p}{p-1}}<n-10.
\]

\item[(ii)] If $f(t)=e^t,\,e^{t^2},\,e^{e^t}$, etc., then $\tau_f(t)\to 1$ as $t\to+\infty$, hence $q_f=5$ and
\[
\dim_{\mathcal{H}} \sing(u)\le n-10.
\]

\item[(iii)] If $f(t)=(1-t)^{-p}$ with $p>0$, then $T_f=1$ and $\tau_f(t)\to (p+1)/p$ as $t\uparrow 1$, so
\[
q_f=1+2\sqrt{\frac{p}{p+1}} +2\frac{p}{p+1},
\]
and therefore
$$
\dim_{\mathcal{H}} \sing(u)\le n-2-4\frac{p}{p+1}-4\sqrt{\frac{p}{p+1}}<n-2.
$$

\item[(iv)] If $f(t)=\exp((1-t)^{-p})$ ($p>0$), $\exp\exp\!\big(\frac{1}{1-t}\big)$, etc., then $T_f=1$ and $\tau_f(t)\to 1$ as $t\uparrow 1$, hence $q_f=5$ and
\[
\dim_{\mathcal{H}} \sing(u)\le n-10.
\]
\end{itemize}
Notice that the exponent $q_f$ is
$$q_f=1+2\liminf_{t\uparrow T_f} \big(\alpha_f(t)+\beta_f(t)\big) \geq 1+2\liminf_{t\uparrow T_f} \alpha_f(t).$$ 
where we defined the two functions
\[
\alpha_f(t):=\frac{\log f(t)}{\log f'(t)}
\qquad\text{and}\qquad
\beta_f(t):=\frac{\int^t_{m_f} \sqrt{{f''(s)}/{f(s)}}\,ds}{\log f'(t)}.
\]
Also,
if $g'\sim f'$ and both $f$ and $f'$ blow up as $t \to T_f$, then $\liminf_{t\uparrow T_f} \alpha_f(t)=\liminf_{t\uparrow T_f} \alpha_g(t)$. Based on this observation, the estimates above yield the following:
\begin{itemize}
\item If $g'\sim f'$ with $f$ as in model cases (i) or (iii), then
\[
q_g\geq 1+2\frac{p}{p-1},
\qquad \dim_{\mathcal{H}} \sing(u)\le n-2-4\frac{p}{p-1}.
\]
\item If $g'\sim f'$ with $f$ as in model cases (ii) or (iv), then
\[
q_g\geq 3,
\qquad \dim_{\mathcal{H}} \sing(u)\le n-6.
\]

\end{itemize}
\end{example}

\subsection{The two-dimensional case}\label{sec:2D}
As anticipated, in dimension two we are able to get a very precise result.
The key is the following universal interior bound for \(f(u)\), whose proof relies on a localized second order delayed-ODE argument for averages of \(f(u)\).
\begin{proposition}
\label{prop:supf-2D}
    Assume $(u,f)\in\sol(B_1)$ with $B_1\subset \R^2$. Then
    \begin{equation}
    f(u)\le C\, \|u\|_{L^\infty(B_{1})}
    \qquad\text{a.e. in }B_{3/4},
    \end{equation}
    where $C$ is a purely dimensional constant.
\end{proposition}
Previously, this result was known only in the radial case, see \cite[Theorem 1.1]{LuoYeZhou}.

\begin{theorem}\label{thm:C11-2D}
Let $(u,f)\in\sol(B_1)$ with $B_1\subset\R^2$. Then
\[
\|D^2u\|_{L^\infty(B_{1/2})}\le C\,\|u\|_{L^1(B_1)},
\]
where \(C\) is universal. Moreover,
\begin{equation}
	\text{ either $\sing(u)= \emptyset$,}\quad\text{ or $T_f<+\infty$ and $f(T_f)<+\infty$.}
\end{equation}
\end{theorem}
The following example shows that the $C^{1,1}$ regularity is optimal in dimension two: indeed one can build examples with $f(T_f)<+\infty$ and $f'(T_f)=+\infty$. In particular, $\sing(u)\neq \emptyset$ and $u\notin C^{2,\eps}$ for any $\eps>0$. We explain in \Cref{rmk:weiss-monneau} below why we do not expect $C^2$ regularity to hold in such generality, drawing a connection with the regularity theory of singular points in the obstacle problem.
\begin{remark}\label{rmk:finitef}
For $n\ge2$, consider the radial solution of
\[
u''+(n-1)\tfrac1r u' = -f(u),\qquad u(0)=1,\qquad u'(0)=0,
\]
with, for some $a>0$, and $t_0$ close to 1,
\[
f(t):= 1 - a\big(\log(\tfrac{1}{1-t})\big)^{-1} \quad\text{for all }t\in[t_0,1).
\]
One can check that, for $t_0$ sufficiently close to $1$, it holds
\[
f\in\NN,\quad f(1-)=1,\quad f'(t)=\tfrac{a}{1-t}\big(\log(\tfrac{1}{1-t})\big)^{-2}\to+\infty \text{ as }t\uparrow 1,
\]
so $T_f=1$. One checks that
\[
u(r)=1-\big(\tfrac{1}{2n}-O(\log(1/r)^{-1})\big)r^2,
\]
hence $1-u(r)\sim r^2$, and therefore
\[
f'(u(r))\sim 4a\, r^{-2} \big(\log(r^{-2})\big)^{-2}\qquad\text{ as }r\downarrow0.
\]
By comparison with Hardy's inequality, $(u,f)\in \sol(B_{r_0})$ for $r_0$ small enough and $a<1/4$ (cf. \eqref{eq:2Dhardy} above). 
\end{remark}

\subsection{Comparison with the existing literature}\label{sect:compare}

The literature for \eqref{eq:PDE} is abundant; we recall here only previous partial regularity results.
Most works are restricted to particular nonlinearities and determine a critical dimension $n_f$ below which all stable solutions are smooth,
but {\it do not} provide a bound of the form $\dim_{\mathcal{H}}\sing(u)\le n-n_f$ (with the exception of a few specific model nonlinearities).

\begin{itemize}
\item For $f(t)=e^t$ (Liouville equation), Wang \cite{wangexp,wangerratum} proved an $\eps$-regularity theorem and $L^{5-}$ bounds for $e^u$, obtaining that the singular set has dimension at most $n-10$ (this corresponds to our Example~\ref{rem:model}(ii)).
Earlier partial regularity results were obtained by Da Lio for stationary solutions in dimension $3$ \cite{dalio}.

\item For $f(t)=t^p$ (Lane--Emden equation), partial regularity results (covering finite Morse index)
were proved by Pacard \cite{pacard}, D\'avila--Dupaigne--Farina \cite{davilapower}, and Wang \cite{wangpower}.

\item For $f(t)=(1-t)^{-p}$ (MEMS-type), Mignot--Puel \cite{mignotpuel} identified the critical dimension in terms of $p$ (the critical dimension coincides with the formula obtained in Example~\ref{rem:model}(iii)).

\item For $f$ satisfying a Crandall--Rabinowitz-type condition
\[
f''(t)f(t) \ge \gamma_-\, f'(t)^2
\quad\text{for all $t$ close to $T_f$}
\]
the best results in the regular case are due to Ye--Zhou \cite{DongFeng}, which we recover and extend to the singular case 
(see \Cref{rmk:lapfp}).
For singular nonlinearities, we also recover the recent results of Bruera--Cabr\'e \cite{cabrebruera}
and the related results of Luo--Ye--Zhou \cite{LuoYeZhou}.
\end{itemize}

\section{\texorpdfstring{$\eps$}{}-regularity results}

\subsection{Notation}

Throughout this section, unless otherwise specified, balls are centered at the origin:
$B_r := B_r(0)$. (All statements are translation invariant, so centering at $0$ causes no loss
of generality.)

We will use the averaged quantities
\begin{equation*}
    \avg_r(u):=\fint_{B_r} u \,dx ,\qquad 
    \osc_{r}(u):=\fint_{B_r}\bigl|u-\avg_r(u)\bigr|\,dx.
\end{equation*}
Note that $\osc_r(u)$ is the \emph{mean} oscillation, rather than an $L^\infty$
oscillation. This is important since $u$ may be unbounded.

For $M$, $r>0$, we also consider the auxiliary functional
\begin{equation}\label{eq:defQ}
Q^M_r(u):= \avg_r(u) + M \cdot \osc_r(u).
\end{equation}
Heuristically, $Q_r^M(u)$ is an ``upper Campanato envelope'' of $u$ at scale $r$.

Given a function $w$ and a scale $r>0$, we denote by
\[
w_r(x):=w(rx)
\]
its rescaling. Then, by a change of variables, one has
\[
\avg_r(u) = \avg_1(u_r),\qquad \osc_r(u) = \osc_1(u_r).
\]
Also, if $-\Delta u=f(u)$ in $B_1$, then $-\Delta u_r = r^2 f(u_r)$ in $B_{1/r}$, so the natural
rescaling of the nonlinearity is $f \mapsto r^2 f$.

\subsection{Oscillation decay lemma}

The next lemma is the key compactness input: it says that if the set where $f'(u)$ is large
has very small measure at a certain scale, then the mean oscillation of $u$ decays going to a smaller scale. This is
a De~Giorgi/Campanato-type mechanism adapted to the stability inequality.

\begin{lemma}\label{lem:decay}
Given $\lambda>0$ and $\theta\in(0,\tfrac13)$ there is $\eta=\eta(n,\theta,\lambda)>0$ with the following property.
If $(u,f)\in \sol(B_1)$ and
\begin{equation}\label{eq:epsass}
|\{f'(u)\ge \lambda\}\cap B_{1/2}|\le \eta,
\end{equation}
then
\[
\osc_\theta(u) \le C(n)\,(1+\lambda)^{n+1}\,\theta\cdot  \osc_1(u).
\]
\end{lemma}

\begin{remark}\label{rmk:markov}
A convenient sufficient condition for \eqref{eq:epsass} is given by Markov/Chebyshev inequality:
if
\[
\int_{B_{1/2}} f'(u)\,dx \le \eta\,\lambda,
\]
then \eqref{eq:epsass} holds.
\end{remark}

\begin{proof}[Proof of \Cref{lem:decay}]
We argue by contradiction. Then there exist $(u_k,f_k)\in\sol(B_1)$ and $\eta_k\downarrow 0$ such that
\begin{equation*}
|\{f_k'(u_k)\ge \lambda\}\cap B_{1/2}|\le \eta_k
\quad\text{but}\quad
\osc_\theta(u_k)>  C (1+\lambda)^{n+1}\theta\,\osc_1(u_k),
\end{equation*}
where $C$ will be chosen (depending only on $n$) at the end.

\medskip
\noindent\textbf{Step 1: Normalization.}
Set
\[
a_k:=\avg_1(u_k),\qquad \omega_k:=\osc_1(u_k),\qquad 
v_k(x):=\frac{u_k(x)-a_k}{\omega_k}.
\]
Then $\avg_1(v_k)=0$ and $\osc_1(v_k)=1$, i.e.
\begin{equation}\label{eq:contradictionlem1}
\fint_{B_1}|v_k|=1,\qquad \osc_\theta(v_k)>C (1+\lambda)^{n+1}\theta.
\end{equation}
Define the rescaled nonlinearities
\[
g_k(s):=\frac{f_k(a_k + \omega_k s)}{\omega_k},\qquad s\in\R.
\]
This rescaling ensures $(v_k,g_k)\in\sol(B_1)$ and a direct computation shows that the stability potential is unchanged:
\[
g_k'(v_k)= f_k'(u_k) \quad \text{in }B_1.
\]

\medskip
\noindent\textbf{Step 2: Limit problem.}
Since $g_k(v_k)\ge 0$, each $v_k$ is superharmonic. Using compactness of superharmonic functions
(e.g., \cite[Lemma A.1]{actadim9}), after extracting a subsequence we have
\[
v_k\to v_*\quad\text{in }L^1_{\loc}(B_1),
\]
and by \Cref{thm:compactness} there exists $g_*\in\NN$ such that $(v_*,g_*)\in\sol(B_1)$ and
$g_k\to g_*$ locally uniformly below $\sup v_*$.

Because $\avg_1(v_k)=0$ and $\fint_{B_1}|v_k|=1$, passing to the limit gives
$\avg_1(v_*)=0,$ $\fint_{B_1}|v_*|\le1$ and $\osc_\theta(v_*)\ge C (1+\lambda)^{n+1}\theta$, hence $v_*$ is {not} constant. Therefore we can use
\eqref{eq:specification} in \Cref{thm:compactness} to get that,
for a.e.\ $x\in B_1$ and every fixed $\delta>0$,
\begin{equation}\label{eq:13clean}
g_k(v_k(x)) \to g_*(v_*(x))
\quad\text{ and }\quad
g_k(v_k(x)-\delta) \to g_*(v_*(x)-\delta)\quad\text{ as }k\to+\infty.
\end{equation}

\medskip
\noindent\textbf{Step 3: $g_*$ is $\lambda$-Lipschitz.}
Fix $\delta>0$. Since $g_k$ is convex, for every $t$ one has
\[
\frac{g_k(t)-g_k(t-\delta)}{\delta}\le g_k'(t),
\]
hence
\[
\{g_k(v_k)-g_k(v_k-\delta)\ge \lambda\delta\}\cap B_{1/2}
\subset
\{g_k'(v_k)\ge \lambda\}\cap B_{1/2}
=
\{f_k'(u_k)\ge\lambda\}\cap B_{1/2}.
\]
Recalling \eqref{eq:epsass}, this implies
\[
|\{g_k(v_k)-g_k(v_k-\delta)\ge \lambda\delta\}\cap B_{1/2}|
\le \eta_k\to 0,
\]
that together with \eqref{eq:13clean} gives
\[
g_*(v_*(x)) - g_*(v_*(x)-\delta)\le \lambda\delta\quad\text{for a.e.\ }x\in B_{1/2}.
\]
Letting $\delta\downarrow 0$ and recalling that $g_*'$ denotes the left derivative, we obtain
\[
g_*'(v_*(x))\le \lambda\quad\text{for a.e.\ }x\in B_{1/2}.
\]

\medskip
\noindent\textbf{Step 4: Elliptic bootstrap.}
By the mean value inequality for superharmonic functions, there exists $C_0=C_0(n)$ such that
\[
v_* \ge -C_0\quad\text{in }B_{1/2}.
\]
Since $g_*$ is nondecreasing, this gives $g_*(v_*)\ge g_*(-C_0)$ in $B_{1/2}$.
Choose $\xi\in C_c^\infty(B_{3/4})$ with $\xi\equiv 1$ on $B_{1/2}$ and $0\le \xi\le 1$.
Testing $-\Delta v_*=g_*(v_*)$ against $\xi$ yields
\[
g_*(-C_0)\,|B_{1/2}|
\le
\int_{B_{1/2}} g_*(v_*)\,dx
\le
\int g_*(v_*)\,\xi\,dx
=
-\int (\Delta v_*)\,\xi\,dx
=
-\int v_*\,\Delta \xi\,dx
\le C_1(n),
\]
where we used $\fint_{B_1}|v_*|=1$ and the boundedness of $\Delta\xi$.
Therefore,
\[
0\le g_*(-C_0)\le C_1(n).
\]
On the other hand, since $g_*'$ is bounded by $\lambda$ a.e.\ on the range of $v_*$ in $B_{1/2}$ (recall Step 3),
it follows that $g_*$ is $\lambda$-Lipschitz there, hence for $x\in B_{1/2}$,
\[
0\le g_*(v_*(x)) \le g_*(-C_0) + \lambda\,(v_*(x)+C_0)
\le C(n)(1+\lambda)\,(1+|v_*(x)|).
\]
As a consequence,
\begin{equation}\label{eq:DeltaLp-clean}
\|\Delta v_*\|_{L^p(B_r)}
=
\|g_*(v_*)\|_{L^p(B_r)}
\le C(n,p)\,(1+\lambda)\,\bigl(1+\|v_*\|_{L^p(B_r)}\bigr),
\qquad \forall\,r\in[\tfrac13,\tfrac12],\ \forall p\ge 1.
\end{equation}
We now bootstrap integrability. Fix equally spaced radii
\[
\tfrac12=r_0>r_1>\cdots>r_{n+1}=\tfrac13.
\]
Start from $v_*\in L^1(B_{1/2})$.
By the $L^1$ Poisson estimate, for any $q\in\bigl(1,\frac{n}{n-2}\bigr)$, with the usual convention that \(\frac{n}{n-2}=+\infty\) when \(n=2\),
\[
\|v_*\|_{L^{q}(B_{r_1})}
\le C(n,q)\Big(\|\Delta v_*\|_{L^1(B_{1/2})}+\|v_*\|_{L^1(B_{1/2})}\Big)
\le C(n,q)(1+\lambda)\Big(1+\|v_*\|_{L^1(B_{1/2})}\Big),
\]
where we used \eqref{eq:DeltaLp-clean} with $p=1$.

Now set $p_1:=q$ and, as long as $2p_i<n$, define for $i\ge1$ the Sobolev exponents
\[
p_{i+1}:=\frac{n p_i}{n-2p_i} \qquad \implies\qquad p_{i+1}=\frac{n p_1}{n-2i p_1}.
\]
Interior elliptic regularity and Sobolev embedding give, for $i\ge 1$,
\[
\|v_*\|_{L^{p_{i+1}}(B_{r_{i+1}})}
\le C(n,p_i)\|v_*\|_{W^{2,p_i}(B_{r_{i+1}})}
\le C(n,p_i)\Big(\|\Delta v_*\|_{L^{p_i}(B_{r_i})}+\|v_*\|_{L^{p_i}(B_{r_i})}\Big).
\]
Using \eqref{eq:DeltaLp-clean} to bound $\|\Delta v_*\|_{L^{p_i}}$ in terms of $\|v_*\|_{L^{p_i}}$,
we obtain
\[
\|v_*\|_{L^{p_{i+1}}(B_{r_{i+1}})}
\le C(n,p_i)\|v_*\|_{W^{2,p_i}(B_{r_{i+1}})} \leq C(n,p_i)\,(1+\lambda)\,\bigl(1+\|v_*\|_{L^{p_i}(B_{r_i})}\bigr).
\]
Iterating this bound until $i=m$ with $p_{m}<n/2<p_{m+1}$, it follows that
$$
\|v_*\|_{W^{2,p_{m+1}}(B_{r_{m+2}})} \leq C(n)(1+\lambda)^{m+2}\,\bigl(1+\|v_*\|_{L^{1}(B_{1/2})}\bigr) \leq C(n)\,(1+\lambda)^{m+2},
$$
hence $v_*$ becomes bounded in $B_{r_{m+2}}$.
Plugging boundedness back into \eqref{eq:DeltaLp-clean} yields $\Delta v_*\in L^\infty(B_{r_{m+2}})$.
Therefore elliptic estimates imply $v_*\in W^{2,p}(B_{r_{m+3}})$ for every $p<+\infty$, and choosing $p>n$
gives, by Sobolev--Morrey,
\[
\|\nabla v_*\|_{L^\infty(B_{r_{m+3}})}\le C(n)\,(1+\lambda)^{m+3}.
\]
Noticing that, by choosing $p_1$ very close to $\frac{n}{n-2}$ (or arbitrarily large if \(n=2\)), the exponent $m+2$ is bounded by $n$, we conclude that
\[
\|\nabla v_*\|_{L^\infty(B_{1/3})}\leq  C(n)\,(1+\lambda)^{n+1}.
\]
Finally, since
\[
\osc_\theta(v_*) \le \fint_{B_\theta}|v_*-\avg_\theta(v_*)|
\le \sup_{B_\theta}|v_*-\avg_\theta(v_*)|
\le \theta\,\|\nabla v_*\|_{L^\infty(B_{1/3})},
\]
we conclude
\[
\osc_\theta(v_*) \le \bar C(n)(1+\lambda)^{n+1}\,\theta.
\]
This contradicts \eqref{eq:contradictionlem1} provided $C$ is chosen larger than $\bar C(n)$,
and completes the proof.
\end{proof}

\subsection{A Campanato-type criterion in terms of \texorpdfstring{$Q_r^M$}{}}

The next corollary packages the previous decay lemma into a convenient ``self-propagating '' criterion:
if $f'$ is small at the level $Q_r^M(u)$, then $u$ is regular at the center and enjoys a
quantitative H\"older estimate.

\begin{corollary}\label{cor:Qdecay}
For any $\lambda> 0$ and $\alpha\in(0,1)$, there is $M=M(n,\alpha,\lambda)>0$ with the following property.
If $(u,f)\in \sol(B_1)$ and
\begin{equation}\label{eq:Qass}
    r^2 f'\big(\avg_r(u)+ M\osc_r(u)\big)< \lambda \quad\text{ for some }r\in(0,1),
\end{equation}
then $0\in \reg(u)$. More precisely, $f'(u(0))<\lambda/r^2$ and
\begin{equation}\label{eq:67}
\fint_{B_\rho} |u -u(0)|\le C \left(\frac{\rho}{r}\right)^\alpha\osc_r(u)
\quad \text{ for all } \rho\in(0,r),
\end{equation}
for some constant $C=C(n,\alpha,\lambda)$.
\end{corollary}

\begin{remark}\label{rmk:open-reg}
Condition \eqref{eq:Qass} is open with respect to small translations: it remains true for
$u(\cdot+z)$ provided $|z|\ll r$. This is crucial to guarantee that the regular set will be open.
\end{remark}


\begin{proof}[Proof of \Cref{cor:Qdecay}]
Fix $\alpha\in(0,1)$. Choose $\theta\in(0,1/3)$ (depending on $n,\lambda,\alpha$) so small that
\[
C(n)(1+\lambda)^{n+1}\,\theta \le \theta^\alpha.
\]
Let $\eta$ be the constant given by \Cref{lem:decay} with these choices of $n,\theta,\lambda$.

\medskip
\noindent\textbf{Step 1: Normalization.}
Consider $(v,g)\in\sol(B_1)$ defined by
\[
v(x):=\frac{u(rx)-\avg_r(u)}{\osc_r(u)},
\qquad
g(s):=r^2\frac{f(\avg_r(u) +s\cdot \osc_r(u))}{\osc_r(u)}.
\]
Then $g'(v)=r^2 f'(\avg_r(u)+ v\,\osc_r(u))$.

Recalling notation \eqref{eq:defQ}, assumption \eqref{eq:Qass} says $r^2 f'(Q_r^M(u))<\lambda$. Since $f'$ is nondecreasing,
if $v(x)<M$ one has $\avg_r(u)+v(x)\,\osc_r(u) < Q_r^M(u)$ and therefore $g'(v(x))<\lambda$.
Hence,
\[
\{ g'(v)\ge \lambda \} \subset \{ v\ge M\}.
\]
By Markov's inequality and the normalization $\osc_1(v)=1$,
\[
|\{ g'(v)\ge \lambda \}|
\le |\{ v\ge M\}|
\le \frac{1}{M}\int_{B_1}|v|
= \frac{|B_1|}{M}.
\]
Thus \Cref{lem:decay} applies provided $M\ge |B_1|/\eta$, and yields
\[
\osc_\theta(v)\le \theta^\alpha \osc_1(v)=\theta^\alpha.
\]
Rescaling back, this becomes
\[
\osc_{\theta r}(u)\le \theta^\alpha \osc_r(u).
\]

\medskip
\noindent\textbf{Step 2: Decay and control $Q_r^M$.}
Repeating the same argument at the smaller scales (using the fact that $Q^M$ is designed to
propagate across scales), we obtain for all $j\ge 0$,
\[
\osc_{\theta^{j+1}r}(u)\le \theta^\alpha \osc_{\theta^{j}r}(u),
\qquad\text{hence}\qquad
\osc_{\theta^{j}r}(u)\le \theta^{\alpha j}\osc_r(u).
\]
Moreover, for the averages we have
\[
|\avg_{\theta r}(u)-\avg_r(u)|
\le \fint_{B_{\theta r}}|u-\avg_r(u)|
\le \frac{|B_r|}{|B_{\theta r}|}\fint_{B_r}|u-\avg_r(u)|
= \theta^{-n}\osc_r(u).
\]
Therefore,
\begin{align*}
Q^{M}_{\theta r}(u)
&=\avg_{\theta r}(u)+M\osc_{\theta r}(u)\\
&\le \avg_r(u)+\theta^{-n}\osc_r(u)+M\theta^\alpha\osc_r(u)
\le Q^M_r(u),
\end{align*}
provided $M$ was chosen so big that $\theta^{-n}+M\theta^\alpha\le M$.

Since $u(0)=\lim_{\rho\downarrow 0}\avg_\rho(u)$ and $\avg_{\theta^j r}(u)\le Q^M_{\theta^j r}(u)\le Q^M_r(u)$,
we get $u(0)\le Q^M_r(u)$, hence by monotonicity of $f'$,
\[
f'(u(0))\le f'(Q^M_r(u))<\frac{\lambda}{r^2}.
\]

\medskip
\noindent\textbf{Step 3: Campanato estimate.}
The geometric decay of $\osc_{\theta^j r}(u)$ implies the Campanato estimate \eqref{eq:67}:
for $\rho\in(0,r)$ choose $j$ such that $\theta^{j+1}r<\rho\le \theta^j r$. Then
\[
\fint_{B_\rho}|u-u(0)|
\le C \fint_{B_{\theta^j r}}|u-u(0)|
\le C \osc_{\theta^j r}(u)
\le C \theta^{\alpha j}\osc_r(u)
\le C\left(\frac{\rho}{r}\right)^\alpha \osc_r(u).
\]
This proves \eqref{eq:67} and completes the proof.
\end{proof}

\subsection{Proofs of \texorpdfstring{\Cref{thm:open}}{} and \texorpdfstring{\Cref{thm:epsreg}}{}}
We first prove that the regular set is always open in the class $\sol$.
\begin{proof}[Proof of \texorpdfstring{\Cref{thm:open}}{}]
Fix $(u,f)\in\sol(B_1)$ and assume $f'(u(0))<+\infty$. Our aim is to show that $f'(u)<+\infty$ in a neighbourhood of the origin. 

By monotonicity of $f'$,
we can choose some $\delta>0$ so that 
$$\lambda:=1+ f'(u(0)+2\delta)<+\infty.$$ 
Let $M=M(n,1/2,\lambda)$ be the constant from \Cref{cor:Qdecay} with $\alpha=\tfrac12$ and this choice of $\lambda$.

By the mean value property for superharmonic functions, $\avg_r(u)\uparrow u(0)$ as $r\downarrow 0$.
Thus, to guarantee \eqref{eq:Qass} for small $r$, it is enough to know that
\begin{equation}\label{eq:oscvanish}
\osc_r(u)\to 0\quad\text{as }r\downarrow 0.
\end{equation}
Indeed, if \eqref{eq:oscvanish} holds, then for $r$ small enough we have
\[
Q^M_r(u)=\avg_r(u)+M\osc_r(u) < u(0)+\delta,
\]
so $f'(Q^M_r(u))\le f'(u(0)+\delta)<\lambda$, hence (since $r<1$)
\[
r^2 f'(Q^M_r(u))<\lambda,
\]
and \Cref{cor:Qdecay} implies $0\in\reg(u)$ and provides a H\"older estimate at the origin.
This also gives openness of $\reg(u)$ by translation (see \Cref{rmk:open-reg}),
and $C^{2,\alpha}$ regularity for $\alpha<1$  inside $\reg(u)$ then follows by standard elliptic regularity.

\smallskip
It now remains to prove \eqref{eq:oscvanish}.
To this aim, we show that blow-ups of $u$ at $x=0$ converge to a constant, which is equivalent to the vanishing
of the mean oscillation.

Let $r_j\downarrow 0$ and set $u_j:=u_{r_j}=u(r_j\,\cdot)$. Given $R>0$,
if $j$ is large enough so that
$Rr_j\le 1/2$ we have
\[
u_j \ge \inf_{B_{1/2}}u\quad\text{in }B_R.
\]
Therefore, since $u$ is superharmonic (and therefore $u_-$ is subharmonic),
for $j$ large (depending on $R$) we have 
\begin{align}\label{eq:00}
\fint_{B_R} |u_j|
&\le \fint_{B_R} u_j -\inf_{B_{1/2}}u +\big|\inf_{B_{1/2}} u \big|\notag\\
&= \fint_{B_R}u_j + 2\big(\inf_{B_{1/2}} u\big)_-\notag\\
&\le u(0) + 2\sup_{B_{1/2}}u_- \le u(0) +C(n)\|u\|_{L^1(B_1)}.
\end{align}
By compactness of superharmonic functions, passing to a diagonal subsequence we may assume
$u_j\to v$ in $L^1_{\loc}(\R^n)$ and a.e.
By \Cref{thm:compactness}, $(u_j,r_j^2 f)\to (v,g)\in \sol(\R^n)$.
But for every fixed $t<\sup_{\R^n}v$, we have $r_j^2 f(t)\to 0$, hence necessarily $g\equiv 0$
on $(-\infty,\sup v)$, and therefore $v$ is harmonic in $\R^n$.

Finally, \eqref{eq:00} implies that $\fint_{B_R} |v|\leq C$ for every $R>0$,  therefore $v$ must be constant by Liouville's theorem.
This proves that every blow-up limit is constant, which is equivalent to \eqref{eq:oscvanish}.
\end{proof}
We turn to the proof of our $\eps$-regularity result for regular nonlinearities, i.e., $T_f=+\infty$.
\begin{proof}[Proof of \Cref{thm:epsreg}]
Assume \eqref{eq:limsup} holds at $x_0=0$ and let $\eps>0$ be chosen (depending only on $n$) so small
that \Cref{lem:decay} applies with $\lambda=1$ and yields a contraction factor $\theta^\alpha$.

Since
\[
\limsup_{r\to 0} r^{2-n}\int_{B_r} f'(u)<\eps,
\]
there exists $r_0>0$ such that for all $r\in(0,r_0)$,
\begin{equation}\label{eq:kl}
   r^{2-n}\int_{B_r} f'(u)\,dx <\eps. 
\end{equation}
Fix $r\in(0,r_0)$. For the rescaled solution $u_r=u(r\,\cdot)$ we have $-\Delta u_r=r^2f(u_r)$, and we see that
\[
\int_{B_{1/2}} (r^2 f')\bigl(u_r\bigr)\,dx
= r^{2-n}\int_{B_{r/2}} f'(u)\,dx
\le r^{2-n}\int_{B_r} f'(u)\,dx
<\eps.
\]
By \Cref{rmk:markov}, choosing $\eps$ small enough guarantees that $(u_r,r^2f)\in\sol(B_1)$ satisfies \eqref{eq:epsass} with $\lambda=1$. Hence \Cref{lem:decay} yields, for a fixed $\alpha\in(0,1)$,
\[
\osc_{\theta}(u_r)\le \theta^\alpha \osc_1(u_r),
\]
which rescales to
\[
\osc_{\theta r}(u)\le \theta^\alpha \osc_r(u).
\]
Because the smallness assumption \eqref{eq:kl} holds for \emph{all} radii in $(0,r_0)$, we can iterate
and obtain
\[
\osc_{\theta^{j+1}r}(u)\le \theta^\alpha \osc_{\theta^j r}(u)\qquad\forall\,j\ge 0.
\]

Now take $M$ as in \Cref{cor:Qdecay} (with $\alpha=\tfrac12$ and $\lambda=1$). As in the proof of
\Cref{cor:Qdecay}, the oscillation decay implies
\[
Q^M_{\theta s}(u)\le Q^M_s(u)\qquad\text{for }s\in\{r,\theta r,\theta^2 r,\dots\}.
\]
Therefore, for all $N\ge 0$,
\[
(\theta^N r)^2\, f'\bigl(Q^M_{\theta^N r}(u)\bigr)
\le \theta^{2N} r^2\, f'\bigl(Q^M_{r}(u)\bigr).
\]
Since $T_f=+\infty$, the number $f'(Q^M_r(u))$ is finite. Choose $N$ so large that
$\theta^{2N} r^2 f'(Q^M_r(u))<1$. Then
\[
(\theta^N r)^2 f'\bigl(Q^M_{\theta^N r}(u)\bigr)<1,
\]
so \Cref{cor:Qdecay} applies at scale $\theta^N r$ and yields $0\in\reg(u)$.
\end{proof}

\subsection{Proofs of \texorpdfstring{\Cref{thm:C1alpha}}{} and \texorpdfstring{\Cref{cor:assmems}}{}}

We now discuss the ``singular'' case $T_f<+\infty$. The key challenge now is that at a singular point
$u$ may still remain bounded (even $C^{1,\alpha}$), so smallness of the scale-invariant mass of $f'(u)$ 
does not force regularity; instead, we first obtain a quantitative decay for $T_f-u$.

\begin{lemma}\label{lem:decayC1}
For every $\sigma>0$ there are $\eps,\rho \in (0,1)$ with the following property.
If $(u,f)\in\sol(B_1)$, $0\in\sing(u)$, and
\begin{equation}\label{eq:assC1decay}
r^{2-n}\int_{B_r} f'(u)\,dx<\eps \quad \text{ for all }r\in(0,1),
\end{equation}
then
\[
\sup_{B_\rho }(T_f- u) \le \rho^{2-\sigma}\,\sup_{B_1}(T_f-u).
\]
\end{lemma}

\begin{proof}
Assume by contradiction that there exist $(u_k,f_k)\in\sol(B_1)$ and $\eps_k\downarrow 0$ such that
\[
r^{2-n}\int_{B_r} f_k'(u_k)\,dx<\eps_k \quad \forall\,r\in(0,1),
\]
but for some $\rho\in(0,1)$ (to be chosen depending only on $n,\sigma$) one has
\begin{equation}\label{eq:contr}
\sup_{B_\rho}(T_{f_k}-u_k) > \rho^{2-\sigma}\,\sup_{B_1}(T_{f_k}-u_k).
\end{equation}
Set $A_k:=\sup_{B_1}(T_{f_k}-u_k)>0$ and define the normalized functions
\[
v_k(x):=\frac{u_k(x)-T_{f_k}}{A_k}\in[-1,0],\qquad w_k:=-v_k=\frac{T_{f_k}-u_k}{A_k}\in[0,1].
\]
Then $w_k$ is nonnegative and subharmonic (since $\Delta(T_{f_k}-u_k)=f_k(u_k)\ge 0$), hence it satisfies the
mean value inequality.

Also $(v_k,g_k)\in\sol(B_1)$ with
\[
g_k(s):=\frac{f_k(T_{f_k}+A_k s)}{A_k}\in\NN.
\]
Indeed $-\Delta v_k = g_k(v_k)$ and $g_k'(v_k)=f_k'(u_k)$ a.e.  Rescaling, for each $r\in(0,1)$,
$((v_k)_r,r^2 g_k)\in\sol(B_1)$ and
\[
\int_{B_{1/2}} (r^2 g_k')\bigl((v_k)_r\bigr)\,dx
=
r^{2-n}\int_{B_{r/2}} f_k'(u_k)\,dx
\le r^{2-n}\int_{B_r} f_k'(u_k)\,dx
<\eps_k.
\]
Fix $\alpha=\tfrac12$ and apply \Cref{rmk:markov} and \Cref{lem:decay} (with $\lambda=1$) for $k$ large.
We obtain, uniformly in $k$,
\[
\osc_{\theta r}(v_k)\le \theta^\alpha \osc_r(v_k)\qquad\forall\,r\in(0,1).
\]
In particular, since $v_k(0)=0$ (because $0\in\sing(u_k)$), a standard Campanato argument yields
\begin{equation}\label{eq:89clean}
\fint_{B_r} v_k \ge -C r^\alpha \qquad \forall\,r\in(0,1),
\end{equation}
for some $C=C(n,\alpha)$ independent of $k$.

Since $v_k$ are uniformly bounded in $L^\infty(B_1)$ and superharmonic, after extracting a subsequence
we get $v_k\to v_*$ in $L^1_{\loc}(B_1)$ and $(v_*,g_*)\in\sol(B_1)$ for some $g_*\in\NN$.
As in the proof of \Cref{lem:decay}, from the smallness assumptions we also get
\[
g_*'(v_*(x))\le 1\quad\text{for a.e.\ }x\in B_{1/2}.
\]
Hence $g_*(v_*)$ has at most linear growth in $v_*$, and elliptic regularity gives
\[
\|v_*\|_{C^{2,\alpha}(B_{1/2})}\le C(n).
\]
Passing to the limit in \eqref{eq:89clean} yields
\[
\fint_{B_r} v_* \ge -C r^\alpha \qquad \forall\,r\in(0,1).
\]
Since $v_*(0)=0$ and $v_*\le 0$, we deduce that $0$ is a maximum point for $v_*$, hence $\nabla v_*(0)=0$.
Therefore, by Taylor expansion,
\[
0\le -v_*(x)\le C(n)\,|x|^2\qquad\text{for }|x|\le \tfrac12.
\]
Hence, since $-v_k$ is nonnegative and subharmonic, the mean value inequality and the $L^1_{\rm loc}$ convergence of $v_k$ to $v_*$ give
\[
\sup_{B_\rho} |v_k|=\sup_{B_\rho} -v_k \le C(n)\fint_{B_{2\rho}} -v_k
= C(n)\fint_{B_{2\rho}} -v_* + o_k(1)
\le C(n)\rho^2 + o_k(1),
\]
as $k\to+\infty$. 
Recalling \eqref{eq:contr}, this proves
\[
\rho^{2-\sigma}\le C(n)\rho^2 + o_k(1),
\]
and contradiction 
provided we first choose $\rho$ so small that $C(n)\rho^\sigma\le \tfrac12$ and then take $k$ large.
This proves the lemma.
\end{proof}

We turn to the proof of our $\eps$-regularity result for singular nonlinearities, i.e., $T_f<+\infty$.

\begin{proof}[Proof of \Cref{thm:C1alpha}]
Let $\eps$ be the constant given by \Cref{lem:decayC1}.
By the assumption
\[
\limsup_{r\to 0} r^{2-n}\int_{B_r(x_0)} f'(u)\,dx < \eps,
\]
we may translate so that $x_0=0$, and (by the definition of $\limsup$) find $r_*>0$ such that
\begin{equation}\label{eq:65clean}
s^{2-n}\int_{B_s} f'(u)\,dx\le\eps\qquad\forall\,s\in(0,r_*).
\end{equation}
If $0\in\reg(u)$ we are done. Otherwise, if $0\in\sing(u)$, \eqref{eq:65clean} guarantees that,
for every $R\in(0,r_*)$, we can apply \Cref{lem:decayC1} to the rescaled pair
$(u_R,R^2 f)\in\sol(B_1)$.\footnote{Indeed, for every $r\in(0,1)$,
\[
r^{2-n}\int_{B_r} (R^2 f')\bigl(u_R\bigr)\,dx
=
(rR)^{2-n}\int_{B_{rR}} f'(u)\,dx
\le \eps.
\]
} Thus, setting
\[
S(R):=\sup_{B_R}(T_f-u),
\]
we obtain
\[
S(\rho R)\le \rho^{2-\sigma}S(R)\qquad\forall\,R\in(0,r_*).
\]
Iterating this inequality gives
\[
S(\rho^j r_*)\le \rho^{j(2-\sigma)}S(r_*)\qquad\forall\,j\ge0.
\]
For every $s\in(0,r_*)$, choose $j\ge0$ such that $\rho^{j+1}r_*<s\le \rho^jr_*$. Since $S$ is nondecreasing,
\[
S(s)\le S(\rho^j r_*)\le \rho^{-(2-\sigma)}S(r_*)\Big(\frac{s}{r_*}\Big)^{2-\sigma}.
\]
This proves \eqref{eq:alternative}.
\end{proof}
Now we show that $C^{2-\sigma}$ contact of $u$ with the blow-up level $T_f$ with $\sigma$ sufficiently small is in contradiction with inverse-power blow-up rate of $f$.
\begin{proof}[Proof of \Cref{cor:assmems}]
Without loss of generality, $x_0=0$ and $T_f=1$.
Choose \(0<\sigma<2p/(1+p)\), and let \(\eps=\eps(n,\sigma)\) be the constant from \Cref{thm:C1alpha}.
The blow up rate of $f$ (i.e., assumption \eqref{eq:assMEMS}) implies\footnote{Indeed, if by contradiction $\limsup_{t\uparrow 1 }(1-t)^{1+p}f'(t)=0,$ for every $\eta>0$ there would exist $\delta>0$ such that $(1-t)^{1+p}f'(t) \leq \eta$ for $t \in (1-\delta,1)$, and consequently
	$$
	f(t)=f(1-\delta)+\int_{1-\delta}^t f'(s)\,ds \leq f(1-\delta)+\eta \int_{1-\delta}^t (1-s)^{-(1+p)}\,ds\leq f(1-\delta)+\frac{\eta}{p} (1-t)^{-p} \qquad \forall\,t \in (1-\delta,1).
	$$
	This implies $\limsup_{t\uparrow 1 }\, (1-t)^{p}f(t)\leq \frac{\eta}{p}$, which contradicts \eqref{eq:assMEMS} since $\eta$ is arbitrary.} a stronger blow-up rate for $f'$
\[
\limsup_{t\uparrow 1 }(1-t)^{1+p}f'(t)>0.
\]
Hence there exist $c>0$ and a sequence $t_j\downarrow 0$ such that
\[
	f'(1-t_j)\ge c \,t_j^{-1-p}.
\]
If \eqref{eq:alternative} held, there would exist $C>0$ such that
\[
1-u(x)\le C|x|^{2-\sigma}\quad\text{for }|x|\text{ small}.
\]
Set $r_j:=(t_j/C)^{1/(2-\sigma)}\downarrow 0$. Then for all $|x|\le r_j$,
\[
u(x)\ge 1-Cr_j^{2-\sigma}=1-t_j,
\]
and by monotonicity of $f'$,
\[
f'(u(x))\ge f'(1-t_j)\ge c\, t_j^{-1-p}
= c\,C^{-1-p}\,r_j^{-(2-\sigma)(1+p)}.
\]
On the other hand, testing stability with a cutoff in $B_{r_j}$ yields the universal bound
\[
C(n)\ge r_j^{2-n}\int_{B_{r_j}} f'(u)\,dx
\ge c\,C^{-1-p}\,r_j^{2-(2-\sigma)(1+p)}.
\]
For our choice of $\sigma$, 
\[
2-(2-\sigma)(1+p)=\sigma(1+p)-2p<0,
\]
so the right-hand side diverges as $j\to+\infty$, leading to a contradiction. Therefore \eqref{eq:alternative} cannot hold
for such a $\sigma$, meaning that necessarily $0\in\reg(u)$.
\end{proof}

\section{\texorpdfstring{$L^q$}{} estimates on \texorpdfstring{$f'(u)$}{}}

Our goal is to prove \Cref{thm:dimbound} and its consequences. This requires several intermediate steps, which are presented in the following sections.

\subsection{A general a priori inequality}
The starting point is a standard, but extremely flexible, device: we test the equation with a nonlinear function of $u$, and we test stability with another nonlinear function of $u$. By choosing these two functions in a compatible way, the terms containing $|\nabla u|^2$ cancel, yielding an inequality where the main quantities are functions of $u$ that will later need to be optimized.

\begin{lemma}\label{lem:handg}
Let $(u,f)\in \sol(B_1)$ with $\sing(u)=\emptyset,$ and let $h$, $g\in W^{1,1}_\loc((-\infty,T_f))$. Assume that
\begin{equation}\label{eq:handg}
    g'(t) \ge h'(t)^2 \quad \text{for a.e. } t<T_f.
\end{equation}
Then, for all test functions $\zeta\in C^{1,1}_c(B_1)$, it holds
\begin{align}\label{eq:stabilityhg}
    \int_{B_1} \Big(h(u)^2 f'(u) - g(u) f(u)\Big)\zeta^2
    \le 2\int_{B_1} G(u)|\nabla \zeta|^2\,dx
    + \int_{B_1} [2G(u)-h(u)^2]\,\zeta\,\lap\zeta\,dx,
\end{align}
where $G$ is any primitive of $g$, i.e. $G'(t)=g(t)$ a.e.
\end{lemma}

\begin{proof}
We first assume $h$ and $g$ are $C^1$. Testing the weak form of $-\lap u=f(u)$ with $\xi=g(u)\zeta^2$ gives
\[
\int_{B_1} g'(u)\zeta^2|\nabla u|^2\,dx
=\int_{B_1} f(u)g(u)\zeta^2\,dx + \int_{B_1} G(u)\,\lap(\zeta^2)\,dx.
\]
On the other hand, testing stability with $\xi=\zeta\,h(u)$ yields
\begin{align*}
\int_{B_1} f'(u)\,\zeta^2 h(u)^2\,dx &\leq \int_{B_1} |\nabla(\zeta h(u))|^2\,dx\\
&=\int_{B_1}\Big(\zeta^2 h'(u)^2|\nabla u|^2 + h(u)^2|\nabla\zeta|^2
+2\zeta h(u)h'(u)\nabla u\cdot\nabla\zeta\Big)\,dx\\
&=\int_{B_1}\Big(\zeta^2 h'(u)^2|\nabla u|^2 + h(u)^2|\nabla\zeta|^2
+\tfrac12 \nabla(h(u)^2)\cdot\nabla(\zeta^2)\Big)\,dx
\end{align*}
Since $\frac12 \Delta(\zeta^2)=|\nabla \zeta|^2+\zeta\,\lap\zeta$, integrating by parts we obtain
\[
\int_{B_1} f'(u)\zeta^2 h(u)^2\,dx \leq \int_{B_1}\zeta^2 h'(u)^2|\nabla u|^2\,dx
-\int_{B_1} h(u)^2\,\zeta\,\lap\zeta\,dx.
\]
Finally, using \eqref{eq:handg} we can replace $h'(u)^2$ by $g'(u)$ in the right-hand side and combine with the first identity, giving the claimed inequality.

If $h,g\in W^{1,1}_\loc$ only, we regularize them by convolution. Condition \eqref{eq:handg} is preserved under convolution thanks to Jensen's inequality, and the inequality passes to the limit.
\end{proof}

\smallskip

From now on 
\[
\text{we fix $f\in\NN$ satisfying \eqref{eq:loggrowth} and \eqref{eq:weakCR}.}
\]We also assume that $f\in C^2((m_f,T_f))$ and that $f(t)\to+\infty$ as $t\uparrow T_f$.
Up to slightly reducing $\gamma_\circ$, we may assume that there exists $\kappa\ge m_f$ with $\kappa<T_f$ such that
\begin{equation}\label{eq:loggrowth1}
    f'(t) \ge (\log(1+t))^{2+\gamma_\circ}\qquad \forall\,t\in(\kappa,T_f),
\end{equation}
and (enlarging $\kappa$ if needed) that 
\begin{equation}\label{eq:loggrowth1-2}
   f(\kappa)\geq 1\quad \text{ and }\quad f'(\kappa)\geq 1.
\end{equation}
Since \(f'(t)\to+\infty\) as \(t\uparrow T_f\), we also fix this same \(\kappa\) so large that
\begin{equation}\label{eq:loggrowth1-3}
|\log f'(t)|^{2+\gamma_\circ}\le f'(t)\qquad\forall\,t\in(\kappa,T_f).
\end{equation}

\smallskip

Our next goal is to construct the ``best possible'' test functions $h$ and $g$ in terms of $f$ and $f'$, and apply them in \Cref{lem:handg}. To facilitate the analysis, the following section introduces notation that simplifies the computations.

\subsection{A convenient tool: the functions \texorpdfstring{${H_\nu}$}{}}
To streamline the forthcoming computations, we use the notation below.

\begin{definition}\label{def:Hnu}
Let $\kappa$ be as in \eqref{eq:loggrowth1}-\eqref{eq:loggrowth1-2}, and let $\nu$ be a locally integrable function on $[\kappa,T_f)$. We define
\begin{equation*}
H_\nu(t):=\exp\bigg\{\int_{\kappa}^{\max\{t,\kappa\}}  \frac{f'(s)}{f(s)}\, \nu(s)\,ds\bigg\}\qquad\text{for all }t<T_f.
\end{equation*}
\end{definition}

The following facts are easy to check (and we will use them repeatedly):
\begin{itemize}
    \item $H_\nu$ is locally absolutely continuous, positive, and $H_\nu(t)=1$ for $t\le\kappa$.
    \item For $t\in(\kappa,T_f)$,
    \[
        f(t)=f(\kappa)H_1(t),
        \qquad
        f'(t)=f'(\kappa)H_\tau(t),
    \]
    where $\tau:=\dfrac{f''f}{(f')^2}$.
    \item The derivative satisfies
    \begin{equation}\label{eq:Hidentities}
	    H'_\nu = \nu\,\frac{f'(\kappa)}{f(\kappa)}\, H_{\nu+\tau-1}
    \quad\text{a.e. on }(\kappa,T_f).
    \end{equation}
    \item If $\nu\le\tilde\nu$ a.e., then $H_\nu\le H_{\tilde\nu}$ pointwise.
\end{itemize}

\subsection{The main a priori estimate}
A careful optimization of \Cref{lem:handg} suggests choosing $h$ close to\footnote{This heuristic comes from balancing the inequality in \Cref{lem:handg} with the constraint \(g'\ge (h')^2\), and is borne out by the model cases where the sharp integrability exponent is recovered.}
\[
H_{1+\sqrt{\tau}}.
\]
To be more precise, with the exact choice $h=H_{1+\sqrt{\tau}}$, the quantity LHS in \eqref{eq:LHSsmallerRHS} below
would cancel, leaving no coercivity. We will therefore keep a small margin (replacing $\sqrt{\tau}$ by
$\mu\sqrt{\tau}$ with $\mu<1$) so that the left-hand side becomes strictly positive and as coercive as possible,
while still allowing us to choose $g$ so that the structural relation between $g$ and $h$ (namely $g'\ge (h')^2$)
is satisfied. We implement this idea in the next proposition.

Set
\[
\varepsilon_\omega:=\sup\{s\in(0,1):\omega(s)\le 1/2\}>0.
\]

\begin{proposition}\label{prop:aprioribound}
Let $f\in\NN$ satisfy \eqref{eq:weakCR} and \eqref{eq:loggrowth1}--\eqref{eq:loggrowth1-3}, and assume that $f\in C^2((\kappa,T_f))$. Set
\[
\tau(t):=\dfrac{f''(t)f(t)}{f'(t)^2}.
\]
Then, for every $\mu \in [1/2,1)$ and every $(u,f)\in\sol(B_1)$ with $\sing(u)=\emptyset$, one has
\begin{equation}\label{eq:higherintegrability}
    \int_{B_{1/2}\cap\{u>\kappa\}} f'(u)\, f(u)^2\,H_{2\mu \sqrt{\tau}}(u)\,dx
    \le C + C\int_{B_1} |u|\,dx.
\end{equation}
Here, $H_{2\mu \sqrt{\tau}}$ is defined as in \Cref{def:Hnu}, and the constant $C$ depends only on upper bounds for
\begin{equation}\label{eq:Cdependence}
    n,\quad \frac{1}{1-\mu},\quad f(\kappa),\quad f'(\kappa),\quad \frac{1}{\gamma_\circ},\quad \frac{1}{\varepsilon_\omega}.
    \end{equation}
\end{proposition}

\begin{proof}
Throughout the proof, $C$ denotes a large constant depending only on the quantities in \eqref{eq:Cdependence}, while $c$ denotes a small positive constant depending on the same list.

\smallskip
\noindent\textbf{Step 1: Defining $\bm h$ and $\bm g$.}
Fix $\mu \in [1/2,1)$. In the rest of the proof, we specialize the index in \(H_\nu\) to the function
\[
\nu(t):=1+\mu\sqrt{\tau(t)}\ge 1\qquad\text{for }t\in(\kappa,T_f).
\]
Then, a direct computation shows that
\begin{equation}\label{eq:nuproperty}
2\nu+\tau-1-\nu^2
= (1-\mu^2)\tau
\ge (1-\mu)\,\tau.
\end{equation}
Now define
\[
h(t):=H_\nu(t),
\qquad
g(t):=\int_{-\infty}^t (h'(s))^2\,ds,
\qquad
G(t):=\int_{-\infty}^t g(s)\,ds.
\]
Then $h,g\in W^{1,1}_\loc((-\infty,T_f))$ and, by construction, $g'= (h')^2$ a.e.
Applying \Cref{lem:handg} with this choice gives, for any $\zeta\in C^2_c(B_1)$,
\begin{equation}\label{eq:LHSsmallerRHS}
\underbrace{\int_{B_1} (h^2 f' - g f)(u)\,\zeta^2\,dx}_{=: \mathrm{LHS}}
\le
\underbrace{\int_{B_1} G(u)\,\lap(\zeta^2)\,dx - \int_{B_1} h(u)^2\,\zeta\,\lap\zeta\,dx}_{=:\mathrm{RHS}}.
\end{equation}

\smallskip
\noindent\textbf{Step 2: Lower bound for LHS.}
For $t\in(\kappa,T_f)$, using \eqref{eq:Hidentities} and the fact that $h'=0$ on $(-\infty,\kappa]$, we can write
\begin{align*}
h(t)^2 \frac{f'(t)}{f(t)} - g(t)
&= \frac{f'(\kappa)}{f(\kappa)}H_{2\nu+\tau-1}(t)
 - \int_{\kappa}^{t} (h'(s))^2\,ds\\
&= \frac{f'(\kappa)}{f(\kappa)}
 +\frac{f'(\kappa)^2}{f(\kappa)^2}\int_{\kappa}^{t} \big(2\nu+\tau-1-\nu^2\big)H_{2\nu+2\tau-2}\,ds.
\end{align*}
Using \eqref{eq:nuproperty} and that $H_{2\nu+2\tau-2}\ge0$, this yields
\begin{align*}
h(t)^2 \frac{f'(t)}{f(t)} - g(t)
&\ge \frac{f'(\kappa)}{f(\kappa)}
 + (1-\mu)\frac{f'(\kappa)^2}{f(\kappa)^2}\int_{\kappa}^{t} \tau\,H_{2\nu+2\tau-2}\,ds\\
&= \frac{f'(\kappa)}{f(\kappa)}
 + \frac{1-\mu}{f(\kappa)^2}\int_{\kappa}^{t} f''\,f\,H_{2\mu\sqrt{\tau}}\,ds.
\end{align*}
(The last identity corresponds to the equality $\tau\,\frac{f'}{f}=\frac{f''}{f'}$ encoded in the $H$-notation.)

Note that, for $t\le\kappa$, we trivially have $(h^2 f' - g f)(t)\ge f'(t)\ge 0$.
Now define, for $t<T_f$,
\begin{equation}\label{eq:Gamma}
\Gamma_f(t):=\int_{\kappa}^{t} f''\,f'\,f^2\,ds,
\end{equation}
so $\Gamma_f$ is nondecreasing.
Using the identity $f(\kappa)f'(\kappa)\, f''\,f\,H_{2\mu\sqrt{\tau}}
= f''\,f'\,f^2\,H_{2\mu\sqrt{\tau}-1-\tau}$, by an integration by parts we obtain
\begin{align}
\label{eq:boundGamma}
f(\kappa)f'(\kappa)\int_{\kappa}^{t} f''\,f\,H_{2\mu\sqrt{\tau}}\,ds
&=\int_{\kappa}^{t} f''\,f'\,f^2\,H_{2\mu\sqrt{\tau}-1-\tau}\,ds\notag\\
&=\int_{\kappa}^{t} \Gamma_f'\,H_{2\mu\sqrt{\tau}-1-\tau}\,ds\notag\\
&= \Gamma_f(t)\,H_{2\mu\sqrt{\tau}-1-\tau}(t)
 - \int_{\kappa}^{t} \Gamma_f\, H_{2\mu\sqrt{\tau}-1-\tau}'\,ds.
\end{align}
	Since $2\mu\sqrt{\tau}-1-\tau\le 0$, the function
$H_{2\mu\sqrt{\tau}-1-\tau}$ is nonincreasing; hence the last term in \eqref{eq:boundGamma} is nonnegative, and therefore
\[
f(\kappa)f'(\kappa)\int_{\kappa}^{t} f''\,f\,H_{2\mu\sqrt{\tau}}\,ds
\ge \Gamma_f(t)\,H_{2\mu\sqrt{\tau}-1-\tau}(t).
\]
At this point we need a lower bound for $\Gamma_f$, and this is where \eqref{eq:weakCR} enters.

\medskip
\noindent\emph{Claim.} There exists $c_\circ>0$ such that
\begin{equation}\label{eq:Gamma-lower}
\Gamma_f(t) \ge c_\circ\, f'(t)^2 f(t)^2 - f'(\kappa)^2 f(\kappa)^2
\qquad\forall\,t\in(\kappa,T_f).
\end{equation}

\smallskip
Assuming the claim for the moment (whose proof will be given in Step 5 below), we can conclude Step 2.
Indeed, combining all previous inequalities and recalling that
$f'f^2H_{2\mu\sqrt{\tau}-1-\tau} = f(\kappa)f'(\kappa)\, f\,H_{2\mu\sqrt{\tau}}$,
we obtain for $t\in(\kappa,T_f)$
\[
h(t)^2 f'(t) - g(t)f(t)
\ge c\, f'(t)f(t)^2 H_{2\mu\sqrt{\tau}}(t) - C f(t).
\]
Evaluating at $t=u(x)$ and integrating against $\zeta^2$ yields
\begin{equation}\label{eq:LHS-lower}
\mathrm{LHS} \ge
c\int_{\{u\ge\kappa\}} f'(u)f(u)^2 H_{2\mu\sqrt{\tau}}(u)\,\zeta^2\,dx
 - C\int_{B_1} f(u)\,\zeta^2\,dx.
\end{equation}

\smallskip
\noindent\textbf{Step 3: Upper bound for RHS.}
We first prove that
\begin{equation}\label{eq:Gbound}
1+2G(t)\le h(t)^2\qquad\forall\,t\in [\kappa,T_f).
\end{equation}
To prove this fact, we first observe that since $G''=g'=(h')^2$, using \eqref{eq:Hidentities} one obtains 
\[
G''(t)=\frac{f'(\kappa)^2}{f(\kappa)^2}\nu(t)^2\,H_{2\nu+2\tau-2}(t)\qquad\forall\,t \in [\kappa,T_f).
\]
Thus, thanks to \eqref{eq:Hidentities}, \eqref{eq:nuproperty}, and the fact that $H_{(1-\mu)\tau}$ is nondecreasing (since $\mu<1$), for $t\in[\kappa,T_f)$ we have
\begin{align*}
	    G'(t)& =\frac{f'(\kappa)^2}{f(\kappa)^2}\int^t_{\kappa}\nu^2 H_{2\nu+2\tau-2}= \frac{f'(\kappa)^2}{f(\kappa)^2}\int^t_{\kappa}(2\nu+\mu^2\tau-1)H_{2\nu+2\tau-2}\\
    &\le \frac{f'(\kappa)^2}{f(\kappa)^2}\int^t_{\kappa}(2\nu+\mu\tau-1)H_{2\nu+(1+\mu)\tau-2} H_{(1-\mu)\tau}= \frac{f'(\kappa)}{f(\kappa)}\int^t_{\kappa}(H_{2\nu+\mu\tau-1})' H_{(1-\mu)\tau}\\
    &\le \frac{f'(\kappa)}{f(\kappa)}H_{(1-\mu)\tau}(t) [H_{2\nu+\mu\tau-1}(t)-1]\le \frac{f'(\kappa)}{f(\kappa)} [H_{2\nu+\tau-1}(t)-1].
    \end{align*}
    On the other hand, since $\nu\ge1$, 
    \begin{align*}
        (h^2)'(t) = H_{2\nu}'(t)&=\frac{f'(\kappa)}{f(\kappa)}{(2\nu)}H_{2\nu+\tau-1}>2\frac{f'(\kappa)}{f(\kappa)} [H_{2\nu+\tau-1}(t)-1]\ge 2G'(t).
    \end{align*}
   This proves that $h^2-2G$ is  nondecreasing on $(\kappa,T_f)$. Hence, since $G(\kappa)=0$ and $h(\kappa)=1$, we conclude that \eqref{eq:Gbound} holds.

Now rewrite RHS as
\begin{align*}
\mathrm{RHS}
&=\int G(u)\lap(\zeta^2)\,dx-\int h(u)^2 \zeta\lap\zeta\,dx\\
&=2\int G(u)|\nabla\zeta|^2\,dx+\int (2G(u)-h(u)^2)\,\zeta\lap\zeta\,dx\\
&\le \int h(u)^2\Big(|\nabla\zeta|^2+|\zeta\lap\zeta|\Big)\,dx,
\end{align*}
where we used the bounds
$$
0\le 2G \leq h^2-1,\qquad |2G-h^2|=h^2-2G\le h^2
$$
which follows from \eqref{eq:Gbound}.

Finally, applying \Cref{lem:logp} to $w=h(u)^2=H_{2\nu}(u)$ with $p=\frac{2+\gamma_\circ}2>1$, we can choose $\zeta\in C^2_c(B_1)$ with $\zeta\ge \chi_{B_{1/2}}$ such that, for any $\eps>0$,
\begin{equation}\label{eq:RHS-est}
\mathrm{RHS}
\le \eps\int_{\{u\ge\kappa\}} H_{2\nu}(u)\,|\log H_{2\nu}(u)|^{2+\gamma_\circ}\,\zeta^2\,dx + C_{\gamma_\circ,\eps}.
\end{equation}

\smallskip
\noindent\textbf{Step 4: Reabsorption and conclusion.}
We now compare $|\log H_{2\nu}(t)|^{2+\gamma}$ with $f'(t)$ for $t\in(\kappa,T_f)$.
Using $\mu\le1$ and the elementary inequality $2+2\sqrt{a}\le 3+a$ for $a\ge0$, we have $2\nu\le 3+\tau$ a.e.
Hence $H_{2\nu}\le H_{3+\tau}$ and
\[
\log H_{2\nu}(t)\le \log H_{3+\tau}(t)
= C + 3\log f(t) + \log f'(t),
\]
for a constant $C$ depending only on $f(\kappa),f'(\kappa)$.
By convexity, $f(t)\le f(\kappa)+f'(t)(t-\kappa)$, so for $t$ large,
\[
\log f(t)\le C+\log f'(t)+\log(1+t),
\]
and therefore
\[
|\log H_{2\nu}(t)|^{2+\gamma_\circ}
\le C\Big(1+|\log f'(t)|^{2+\gamma_\circ}+|\log(1+t)|^{2+\gamma_\circ}\Big).
\]
By \eqref{eq:loggrowth1} and \eqref{eq:loggrowth1-3},
\[
|\log H_{2\nu}(t)|^{2+\gamma_\circ}\le C f'(t)\qquad \forall\,t\in(\kappa,T_f).
\]
Plugging this into \eqref{eq:RHS-est} and then combining \eqref{eq:LHSsmallerRHS} with \eqref{eq:LHS-lower} yields
\[
c\int_{\{u\ge\kappa\}} f'(u)H_{2\nu}(u)\zeta^2\,dx - C\int f(u)\zeta^2
\le C\eps\int_{\{u\ge\kappa\}} f'(u)H_{2\nu}(u)\zeta^2\,dx + C_{\gamma,\eps}.
\]
Choosing $\eps$ so that $C\eps\le c/2$, we reabsorb the main term and obtain \eqref{eq:higherintegrability} (recall that, up to the harmless constant \(f(\kappa)^{-2}\), one has $H_{2\nu}=f^2 H_{2\mu\sqrt{\tau}}$ by the $H$-identities).

\smallskip
\noindent\textbf{Step 5: Proof of claim \eqref{eq:Gamma-lower}.}
We use the identity
\[
2\Gamma_f(t)=\int_{\kappa}^{t}\big[(f')^2\big]' f^2
= f'(t)^2 f(t)^2 - f'(\kappa)^2 f(\kappa)^2 - \int_{\kappa}^{t} (f')^2 (f^2)'.
\]
Fix $t \in (\kappa,T_f)$ and let $s\in(\kappa,t)$ to be specified. Since $f'$ is nondecreasing,
\begin{align*}
\int_{\kappa}^{t} (f')^2 (f^2)'
&=\int_{\kappa}^{s}(f')^2 (f^2)'+\int_{s}^{t}(f')^2 (f^2)'\\
&\le f'(s)^2\int_{\kappa}^{s}(f^2)' + f'(t)^2\int_{s}^{t}(f^2)'\\
&\le f'(s)^2 f(s)^2 + f'(t)^2\big(f(t)^2-f(s)^2\big)\\
&= f'(t)^2 f(t)^2\bigg(1-\frac{f(s)^2}{f(t)^2}+\frac{f'(s)^2 f(s)^2}{f'(t)^2 f(t)^2}\bigg).
\end{align*}
Choose $\varepsilon_\circ\in(0,\varepsilon_\omega)$ such that $\omega(\varepsilon_\circ)\le 1/2$.
If \(\varepsilon_\circ f(t)\le f(\kappa)\), then
\[
\Gamma_f(t)\ge f(\kappa)^2\int_\kappa^t f''f'\,ds
=\frac{f(\kappa)^2}{2}\big(f'(t)^2-f'(\kappa)^2\big),
\]
and \eqref{eq:Gamma-lower} follows, with a smaller \(c_\circ\), from \(f(t)\le f(\kappa)/\varepsilon_\circ\).
We may therefore assume that \(\varepsilon_\circ f(t)>f(\kappa)\). By continuity and monotonicity of $f$ on $(\kappa,T_f)$, we can choose $s=s(t)\in(\kappa,t)$ such that
\[
\frac{f(s)}{f(t)}=\varepsilon_\circ.
\]
Then, by \eqref{eq:weakCR},
\[
\frac{f'(s)}{f'(t)}\le \omega\Big(\frac{f(s)}{f(t)}\Big)=\omega(\varepsilon_\circ)\le\frac12.
\]
Therefore
\[
1-\frac{f(s)^2}{f(t)^2}+\frac{f'(s)^2 f(s)^2}{f'(t)^2 f(t)^2}
\le 1-\varepsilon_\circ^2+\frac14\varepsilon_\circ^2
=1-\frac34\,\varepsilon_\circ^2.
\]
Hence
\[
\int_{\kappa}^{t}(f')^2 (f^2)'
\le \Big(1-\tfrac34\varepsilon_\circ^2\Big) f'(t)^2 f(t)^2.
\]
Plugging back gives
\[
2\Gamma_f(t)\ge \frac34\varepsilon_\circ^2\, f'(t)^2 f(t)^2 - f'(\kappa)^2 f(\kappa)^2,
\]
so \eqref{eq:Gamma-lower} holds with $c_\circ=\frac38\,\varepsilon_\circ^2$.
\end{proof}

\begin{remark}\label{rmk:lapfp}
Let $f\in \NN\cap C^2$ and assume that
\[
\tau(t):=\frac{f''(t)f(t)}{f'(t)^2}\ge \gamma_-\quad\text{ for all }\kappa<t<T_f,
\]
for some constant $\gamma_-\ge 0,$
so that
\[
H_{2\mu\sqrt{\tau}}(u)\ge H_{2\mu\sqrt{\gamma_-}}(u)
=\bigg(\frac{f(u)}{f(\kappa)}\bigg)^{2\mu\sqrt{\gamma_-}}.
\]
Since $\mu$ can be chosen arbitrarily close to $1$, \eqref{eq:higherintegrability} yields
\[
\int_{B_{1/2}} f'(u)\,f(u)^p \le C + C\int_{B_1}|u|<+\infty
\qquad\text{for all }1\le p<2+2\sqrt{\gamma_-}.
\]
Since $-\Delta u=f(u)$ and $p\ge 1$, a direct computation gives
\begin{align*}
-\Delta \big(f(u)^p\big)&= p\,f'(u)\,f(u)^p
 - p(p-1)\,f(u)^{p-2}\,f'(u)^2\,|\nabla u|^2
 - p\,f(u)^{p-1}\,f''(u)\,|\nabla u|^2\le p\,f'(u)\,f(u)^p,
\end{align*}
Since $f'(u)\,f(u)^p\in L^1(B_{1/2})$,
this implies that $\Delta\big(f(u)^p\big)$ is a measure with finite mass in $B_{2/5}$, and by 
interior regularity estimates we deduce that $f(u)\in L^r(B_{1/3})$ for all
\[
r<\frac{2n}{n-2}(1+\sqrt{\gamma_-}).
\]
Thus, recalling that $f(u)=-\Delta u$, elliptic regularity provides an H\"older bound on $u$ inside $B_{1/4}$ as soon as
\begin{equation}\label{eq:dimboundRegular}
    n<6+4\sqrt{\gamma_-}.
\end{equation}
In the singular case $T_f<+\infty$, H\"older continuity is not enough to exclude the presence of singular points (cf \Cref{rmk:salt}), so one argues as in \cite{cabrebruera}: the function $F(u):=\int^u_0f(s)\, ds$ solves $-\Delta(F(u))\le f(u)^2$. If we additionally assume the non-integrability condition $F(T_f)=+\infty$, this gives that no singular solution can exist if $r>n$, that is
\begin{equation}\label{eq:dimboundMEMS}
    n<4+2\sqrt{\gamma_-}.
\end{equation}
\end{remark}

\subsection{Approximation by classical solutions}
We next recall a known sub/supersolution procedure that will be used to approximate possibly singular stable solutions for a certain $f\in \NN$ by bounded stable solutions for the slightly smaller nonlinearity $f_\eps:=(1-\eps)f$. While more direct approximation schemes exist (cf. \cite[Proposition 4.2]{actadim9}), the procedure presented here is necessary in order not to disrupt our asymptotic quantity, namely keeping $q_{f_\eps}=q_f$. 

\begin{lemma}\label{lem:subsuper}
Let $f\in \NN$ and $V\in H^1(B_1)$ with $f(V)\in L^1(B_1)$.
Assume that $V$ is a weak supersolution:
\[
\int_{B_1}\nabla V\cdot\nabla \xi \,dx \ge \int_{B_1} f(V)\,\xi\,dx
\qquad \forall\,\xi\in C^1_c(B_1),\ \xi\ge0.
\]
Then there exists $w\in H^1(B_1)$ such that $w=V$ on $\partial B_1$, $w\le V$ a.e.\ in $B_1$, and
\[
\int_{B_1}\nabla w\cdot\nabla \xi\,dx = \int_{B_1} f(w)\,\xi\,dx
\qquad \forall\,\xi\in C^1_c(B_1).
\]
Moreover,
\[
\|\nabla w\|_{L^2(B_1)} \le 4\, \|\nabla V\|_{L^2(B_1)}.
\]
\end{lemma}

\begin{proof}
Let $w_0$ be the harmonic replacement of $V$ in $B_1$ (i.e.\ $-\lap w_0=0$ in $B_1$ and $w_0=V$ on $\partial B_1$).
Define iteratively $w_{k+1}$ as the weak solution of
\[
-\lap w_{k+1}=f(w_k)\quad\text{in }B_1,\qquad w_{k+1}=V\ \text{on }\partial B_1.
\]
By comparison and monotonicity of $f$ we have
\[
w_0\le w_k\le w_{k+1}\le V\qquad\text{a.e. in }B_1.
\]
Testing the equation for $w_k$ with $w_k-w_0\ge0$ and using $w_k\le V$ yields an energy bound
\[
\|\nabla w_k\|_{L^2(B_1)} \le 2\|\nabla V\|_{L^2(B_1)}+2\|\nabla w_0\|_{L^2(B_1)}
\le 4\|\nabla V\|_{L^2(B_1)}.
\]
Since the sequence $\{w_k\}_{k\geq 0}$ is monotone increasing and bounded above by $V$, it converges a.e.\ to a limit $w\le V$; the uniform $H^1$ bound yields weak $H^1$ convergence along subsequences, and passing to the limit gives the desired weak solution with the stated energy bound.
\end{proof}

\begin{proposition}\label{prop:approximation}
Let $(u,f)\in \sol(B_1)$ with $\sing(u)\cap B_{1/2}\neq\emptyset$ and let $\eps>0$.
Then there exists $u_\eps\in H^1(B_{1/2})$ such that
\[
(u_\eps,(1-\eps)f)\in \sol(B_{1/2})
\qquad\text{and}\qquad
\sing(u_\eps)=\emptyset.
\]
Moreover, as $\eps\downarrow 0$, $u_\eps\uparrow u$ pointwise in $B_{1/2}$ and $u_\eps\to u$ strongly in $H^1(B_{1/2})$.
\end{proposition}

\begin{proof}
Since $\sing(u)\cap B_{1/2}\neq\emptyset$, elliptic regularity forces $\lim_{t\uparrow T_f} f'(t)=+\infty$.
Fix $t_0<T_f$ such that $f(t_0)>0$ and $f'(t_0)>1$, and define the strictly increasing $C^1$ function
\[
h(t):=\int_{t_0}^{t}\frac{ds}{f(s)},\qquad t<T_f.
\]
For $\eps>0$ define $\Phi_\eps$ by
\[
\Phi_\eps(t)=
\begin{cases}
t, & t\le t_0,\\[2mm]
h^{-1}\big((1-\eps)h(t)\big), & t_0<t<T_f.
\end{cases}
\]
For $t>t_0$ the identity $h(\Phi_\eps(t))=(1-\eps)h(t)$ implies
\[
\Phi_\eps'(t)\,f(t)=(1-\eps)\,f(\Phi_\eps(t)).
\]
Since $\Phi_\eps(t)\le t$ and $f'$ is increasing, we have $f'(\Phi_\eps(t))\le f'(t)$ and thus
\[
\Phi_\eps''(t)
=(1-\eps)\frac{f(\Phi_\eps(t))}{f(t)^2}\Big((1-\eps)f'(\Phi_\eps(t))-f'(t)\Big)\le 0,
\]
so $\Phi_\eps$ is concave on $(t_0,T_f)$.
Consequently,
\[
-\lap(\Phi_\eps(u))=\Phi_\eps'(u)\,f(u)-\Phi_\eps''(u)|\nabla u|^2
\ge (1-\eps)f(\Phi_\eps(u)),
\]
meaning that $\Phi_\eps(u)$ is a weak supersolution for the equation $-\lap w=(1-\eps)f(w)$.

Applying \Cref{lem:subsuper} in $B_{1/2}$ with $V=\Phi_\eps(u)$, and taking the monotone-iteration solution constructed in its proof, we obtain $u_\eps\le \Phi_\eps(u)$ in $B_{1/2}$ with
\[
-\lap u_\eps=(1-\eps)f(u_\eps)\quad\text{in }B_{1/2}.
\]
By the minimality of this solution we also get stability of \(u_\eps\), hence \((u_\eps,(1-\eps)f)\in\sol(B_{1/2})\).
If $h(T_f)=\int_{t_0}^{T_f}\frac{ds}{f(s)}<+\infty$, then $\Phi_\eps(t)\to \Phi_\eps(T_f):=h^{-1}((1-\eps)h(T_f))<T_f$ as $t\uparrow T_f$, hence $\Phi_\eps(u)\le \Phi_\eps(T_f)<T_f$ and thus $\sing(u_\eps)=\emptyset$.

If instead $h(T_f)=+\infty$, one can iterate this concave truncation construction a finite number of times (exactly as in \cite[Theorem 3.2.1]{dupaigne}) to obtain a bounded supersolution, and then repeat the argument to conclude again that $\sing(u_\eps)=\emptyset$.

Finally, since $\Phi_\eps(t)\uparrow t$ pointwise as $\eps\downarrow 0$, comparison and the minimal choice of \(u_\eps\) give that $u_\eps$ converges monotonically a.e. to a function $\tilde u \leq u$.
If by contradiction $\tilde u\neq u$, the argument of \cite[Proposition 4.2, Step 3]{actadim9} implies that $f$ must be affine on the family of intervals $[\tilde u(x),u(x)]$ as $x$ varies in $B_1$, forcing $f(u)$ to be affine in $u$ on a nontrivial interval of the form $[t_0,T_f)$; however, this would imply that $u$ is regular, contradicting $\sing(u)\cap B_{1/2}\neq\emptyset$. This proves that $\tilde u=u$.

Finally, the strong convergence of $u_\eps$ to $u$ in $H^1(B_{1/2})$ follows from \Cref{thm:compactness}.
\end{proof}

\subsection{Proofs of \texorpdfstring{\Cref{thm:dimbound}}{} and \texorpdfstring{\Cref{cor:CRvsGamma}}{}}

\begin{proof}[Proof of \texorpdfstring{\Cref{thm:dimbound}}{}]
By a standard covering argument, we can assume that $\Omega=B_1$ and we only need to prove a universal bound around the origin.

Assume $q_f>1$ (otherwise the statement is already contained in the stability $L^1$ bound), and 
fix $q<q_f$.
By definition of $q_f$, choose  $\mu\in[1/2,1)$ close enough to $1$ so that
\[
\frac{q-1}{2}<\mu\liminf_{t\uparrow T_f}\frac{\log f(t)+\int_{m_f}^{t}\sqrt{f''(s)/f(s)}\,ds}{\log f'(t)}.
\]
Since $\log f(t)\ge0$ for $t$ close to $T_f$, up to enlarging $\kappa$ we can ensure that, for all $t\in(\kappa,T_f)$,
\[
2\log f(t) + 2\mu\int_{m_f}^{t}\sqrt{\frac{f''(s)}{f(s)}}\,ds
\ge (q-1)\log f'(t),
\]
which is equivalent to
\[
f'(t)^q
\le f'(t)\,f(t)^2\,\exp\bigg\{2\mu\int_{m_f}^{t}\sqrt{\frac{f''(s)}{f(s)}}\,ds\bigg\}=C_f \, f'(t)\,f(t)^2H_{2\mu\sqrt{\tau}}(t),
\]
where
\[
C_f:=\exp\bigg\{2\mu\int_{m_f}^{\kappa}\sqrt{\frac{f''(s)}{f(s)}}\,ds\bigg\}<+\infty.
\]
We shall use this estimate after multiplying the nonlinearity by fixed positive constants. This causes no difficulty: for \(F=cf\), with \(c>0\), the quantities \(\tau\) and \(q_f\) are unchanged, while \(F'=cf'\), and the corresponding \(H_{2\mu\sqrt{\tau}}\) changes at most by a multiplicative constant after adjusting the base point \(\kappa\). One chooses \(\kappa_c\) so that \(F(\kappa_c),F'(\kappa_c)\ge1\) and the analogues of \eqref{eq:loggrowth1}--\eqref{eq:loggrowth1-3} hold for \(F\). The part of the integral where \(u\le\kappa_c\) is harmless because \(f'\) is bounded there. In the finite covers below the scaling constants \(c\) range over a fixed compact subset of \((0,+\infty)\), and in the approximation step \(c=1-\eps\in[1/2,1]\) after discarding finitely many \(\eps\)'s; hence the constants are uniform for the limiting argument.

If $\sing(u)\cap B_{1/2}=\emptyset$, we cover \(B_{1/4}\) by finitely many balls \(B_\rho(x_i)\) such that \(B_{2\rho}(x_i)\subset B_{1/2}\), and apply \Cref{prop:aprioribound} after rescaling on each \(B_{2\rho}(x_i)\). This gives
\[
\int_{B_{1/4}} f'(u)^q\,dx
\le C_f \int_{B_{1/4}} f'(u)f(u)^2H_{2\mu\sqrt{\tau}}(u)\,dx
\le C + C\int_{B_{1/2}}|u|\,dx,
\]
which proves the result.

If $\sing(u)\cap B_{1/2}\neq\emptyset$, let $u_\eps$ be given by \Cref{prop:approximation} with $f_\eps=(1-\eps)f$.
Then $(u_\eps,f_\eps)\in\sol(B_{1/2})$ and $\sing(u_\eps)=\emptyset$. 
Moreover the assumptions \eqref{eq:loggrowth} and \eqref{eq:weakCR} are uniform in $\eps$ (since they are stable under multiplying $f$ by a constant).
Applying the argument above to $(u_\eps,f_\eps)$ yields uniform $L^q$ bounds on $f'(u_\eps)$ in $B_{1/4}$.
Letting $\eps\to0$ and using Fatou's lemma gives the same bound for $f'(u)$, hence $f'(u)\in L^q(B_{1/4})$.

Finally, the Hausdorff-dimension estimate then follows from \eqref{eq:Lpbound} (i.e.\ from \Cref{thm:epsreg} and \Cref{cor:assmems} together with \Cref{lem:gmtintro}).
\end{proof}

\begin{proof}[Proof of \Cref{cor:CRvsGamma}]
Given $f\in C^2\cap\NN$, define $\tau(t):=\dfrac{f''(t)f(t)}{f'(t)^2}$ and 
\[
\gamma_-:=\liminf_{t\uparrow T_f}\tau(t),
\qquad
\gamma_+:=\limsup_{t\uparrow T_f}\tau(t).
\]
If \(\gamma_-=0\), then the integral term in \eqref{eq:q_f} is nonnegative, and the argument below with the prefactor \(1+\sqrt{\gamma_-}\) replaced by \(1\) gives the desired bound. We therefore assume \(\gamma_->0\).

Then, for every $\delta\in(0,\gamma_-)$ there exists $t_\delta<T_f$ such that
\[
\gamma_- -\delta \le \tau(t)\le \gamma_+ +\delta
\qquad\forall\,t\in(t_\delta,T_f).
\]
For $t>t_\delta$ we estimate
\[
\int_{m_f}^{t}\sqrt{\frac{f''(s)}{f(s)}}\,ds
=\int_{m_f}^{t}\sqrt{\tau(s)}\,\frac{f'(s)}{f(s)}\,ds
\ge \int_{t_\delta}^{t}\sqrt{\tau(s)}\,\frac{f'(s)}{f(s)}\,ds\geq \sqrt{\gamma_- -\delta}\,\Big(\log f(t)-\log f(t_\delta)\Big).
\]
Therefore, recalling \eqref{eq:q_f} and using that $f'(t) \to +\infty$ as $t\uparrow T_f$,
\[
q_f
\ge 1 + 2\liminf_{t\uparrow T_f}\frac{\log f(t)+\sqrt{\gamma_- -\delta}\,\log f(t)}{\log f'(t)}
= 1 + 2(1+\sqrt{\gamma_- -\delta})\,\liminf_{t\uparrow T_f}\frac{\log f(t)}{\log f'(t)}.
\]
Now, by de l'H\"opital rule for the liminf (see for instance \cite{TaylorHopital}), 
$$
\liminf_{t\uparrow T_f}\frac{\log f(t)}{\log f'(t)} \geq \liminf_{t\uparrow T_f}\frac{\big(\log f(t)\big)'}{\big(\log f'(t)\big)'}=\liminf_{t\uparrow T_f}\frac{f'(t)^2}{f''(t)f(t)}=\frac{1}{\gamma_+}.
$$
Combining these bounds and letting $\delta\downarrow 0$, we conclude the validity of \eqref{eq:correctq}.
\end{proof}
\section{The two-dimensional case}
The goal of this section is to prove \Cref{thm:C11-2D}. We recall that, in dimension $n\leq 9$, stable solutions are H\"older continuous \cite{actadim9}, so in particular $u$ and $f(u)$ are pointwise well-defined.

\subsection{A priori estimates for singular nonlinearities in all dimensions}
Our first estimate is a local $L^1$ estimate on $f'(u)f(u).$ This is based on \Cref{lem:handg} and reads as follows:
\begin{proposition}\label{prop:ff'}
    Let $(u,f)\in \sol(B_1)$ with $B_1\subset\R^n$, then
    \begin{equation}\label{eq:goal}
 \int_{B_{1/2}} f'(u)f(u)\,dx
 \le C(n)\, \|u\|_{L^\infty(B_1)}.
\end{equation}
\end{proposition}
\begin{remark}
This bound implies that, if $T_f<+\infty$ or $n\le 9$ (see \cite{actadim9}), then $f'(u)f(u)\in L^1_\loc$.
\end{remark}

To prove the result, we first 
show a preliminary bound in $B_1$ that we will iterate at smaller scales.

\begin{lemma}\label{lem:max-test}
Let $(u,f)\in \sol(B_1)$ with $\sing(u)=\emptyset$. Then, for every $\eps>0$,
\begin{equation}\label{eq:max-test}
 \int_{B_{1/2}} f(u)f'(u)\,dx
 \le \eps\int_{B_1} f(u)f'(u)\,dx
 +\frac{C}{\eps}\|u\|_{L^\infty(B_1)},
\end{equation}
where $C=C(n)$ is dimensional.
\end{lemma}

\begin{proof}
We first record two elementary bounds. By stability, applied with a cutoff equal to one in $B_{1/2}$,
\begin{equation}\label{eq:stability-fprime-bound}
\int_{B_{1/2}} f'(u)\,dx\le C.
\end{equation}
Also, if $\eta\in C_c^2(B_1)$ satisfies $\eta\ge \chi_{B_{2/3}}$, then testing the equation against $\eta$ gives
\begin{equation}\label{eq:f-L1-bound}
\int_{B_{2/3}} f(u)\,dx
= -\int_{B_1} u\,\Delta\eta\,dx
\le C \|u\|_{L^\infty(B_1)}.
\end{equation}
Let $K>0$ be a parameter to be chosen. On the set $\{f(u)\le K\}$, \eqref{eq:stability-fprime-bound} gives
\begin{equation}\label{eq:low-f-bound}
\int_{\{f(u)\le K\}\cap B_{1/2}} f(u)f'(u)\,dx
\le K\int_{B_{1/2}} f'(u)\,dx
\le CK.
\end{equation}
We now estimate the contribution of $\{f(u)>K\}$. If this set is empty, then \eqref{eq:max-test} follows from \eqref{eq:low-f-bound}, after the choice of $K$ below. Otherwise, choose $\kappa<T_f$ such that
\[
f(\kappa)=K,
\]
and apply \Cref{lem:handg} with
\[
h(t):=\max\{f(t), K\},
\qquad
g(t):=\int_{\kappa}^{\max\{t,\kappa\}} f'(s)^2\,ds.
\]
Then \(g'\ge (h')^2\). Moreover, for \(t\ge\kappa\), convexity gives
\[
g(t)\le f'(t)\int_\kappa^t f'(s)\,ds
=f'(t)(f(t)-K),
\]
hence
\[
\bigl(h^2f'-gf\bigr)(u)
\ge K f(u)f'(u)\qquad \text{inside }\{f(u)\ge K\}.
\]
Let $\zeta\in C_c^2(B_{2/3})$ satisfy $\zeta\equiv1$ in $B_{1/2}$, and set
\[
G(t):=\int_{\kappa}^{\max\{t,\kappa\}} g(s)\,ds.
\]
Then 
\[
0\le G(t)\le \frac12(f(t)-K)_+^2,\qquad 
h(t)^2\le C\bigl((f(t)-K)_+^2+K^2\bigr),
\]
and the right-hand side in \Cref{lem:handg} is bounded by
\[
C\int_{B_{2/3}} (f(u)-K)_+^2\,dx + CK^2.
\]
Combining this with the lower bound on the left-hand side yields
\[
K\int_{\{f(u)\ge K\}\cap B_{1/2}} f(u)f'(u)\,dx
\le
C\int_{B_{2/3}} (f(u)-K)_+^2\,dx + CK^2.
\]
After dividing by \(K\) and adding \eqref{eq:low-f-bound}, we obtain
\begin{equation}\label{eq:max-test-preliminary}
\int_{B_{1/2}} f(u)f'(u)\,dx
\le
\frac{C}{K}\int_{B_{2/3}} (f(u)-K)_+^2\,dx
+CK.
\end{equation}
We next bound the square term. On \(\{f(u)>K\}=\{u>\kappa\}\), convexity gives
\[
f(u)-K
\le f'(u)(u-\kappa),
\]
thus
\[
(f(u)-K)_+^2
\le (f(u)-K)_+ f'(u)(u-\kappa)_+
\le f(u)f'(u)\,\|u-\kappa\|_{L^\infty(B_1)}.
\]
Plugging this into \eqref{eq:max-test-preliminary} gives
\begin{equation}\label{eq:max-test-before-K}
\int_{B_{1/2}} f(u)f'(u)\,dx
\le
\frac{C\|u-\kappa\|_{L^\infty(B_1)}}{K}
\int_{B_1} f(u)f'(u)\,dx
+CK.
\end{equation}
It remains to choose \(K\). Let
\[
m:=\inf_{B_1}u
\]
and choose
\[
K:=\frac{A}{\eps}\bigl(f(m)+\|u\|_{L^\infty(B_1)}\bigr),
\]
where \(A\) is a dimensional constant large enough. By \eqref{eq:f-L1-bound} and monotonicity of \(f\),
\[
f(m)\le \fint_{B_{2/3}} f(u)\,dx\le C\|u\|_{L^\infty(B_1)},\qquad \text{so}\quad 
K\le \frac{C A}{\eps}\|u\|_{L^\infty(B_1)}.
\]
Since \(\{f(u)>K\}\) is nonempty, then \(f(\kappa)=K>f(m)\). Hence \(\kappa\ge m\ge -\|u\|_{L^\infty(B_1)}\), while also \(\kappa\le \|u\|_{L^\infty(B_1)}\), therefore
\[
\|u-\kappa\|_{L^\infty(B_1)}\le 2\|u\|_{L^\infty(B_1)}.
\]
Thus, choosing \(A\) sufficiently large,
\[
\frac{C\|u-\kappa\|_{L^\infty(B_1)}}{K}\le \eps.
\]
Using this in \eqref{eq:max-test-before-K}, and recalling the bound on \(K\), gives
\[
\int_{B_{1/2}} f(u)f'(u)\,dx
\le
\eps\int_{B_1} f(u)f'(u)\,dx
+\frac{C}{\eps}\|u\|_{L^\infty(B_1)},
\]
as desired.
\end{proof}

In order to complete the proof of \Cref{prop:ff'}, we recall the following classical result (see \cite{simonscaling}; see also \cite[Lemma~A.4]{actadim9}):
\begin{lemma}\label{lem:covering}
Let \(\mathfrak m\) be a nonnegative function defined on open balls \(B\subset B_1\), finite on compactly contained balls, and assume that \(\mathfrak m\) is subadditive under finite covers, namely
\[
B\subset \bigcup_{j=1}^N B_j
\quad\Longrightarrow\quad
\mathfrak m(B)\le \sum_{j=1}^N \mathfrak m(B_j).
\]
Fix \(\beta\in\R\). There exist
\(\varepsilon_0\), \(C>0\), depending only on \(\beta\) and the dimension $n$, with the following property.
If, for some \(M\ge0\) and every \(B_r(y)\subset B_1\),
\[
 r^\beta\mathfrak m(B_{r/4}(y))
 \le
 \varepsilon_0 r^\beta\mathfrak m(B_r(y))+M,
\]
then $\mathfrak m(B_{1/2})\le CM.$
\end{lemma}
We can now prove the desired estimate.

\begin{proof}[Proof of \Cref{prop:ff'}]
It is enough to establish \Cref{prop:ff'} in the regular case, namely $\sing(u)=\emptyset,$ and then argue by approximation from below using for example \Cref{prop:approximation} (of course if $u\notin L^\infty(B_1)$ there is nothing to prove). 

We define
\[
\mathfrak m(E):=\int_E f(u)f'(u),\qquad M:=\frac{C}{\eps}\|u\|_{L^\infty(B_1)},
\]  
and we apply the inequality at scale one to $u(y+r\,\cdot)$, getting
\[
r^{4-n}\mathfrak m(B_{r/2}(y)) \le \eps r^{4-n} \mathfrak m(B_{r}(y)) + \frac C\eps\|u\|_{L^\infty(B_r(y))}\le \eps r^{4-n} \mathfrak m(B_{r}(y)) + M.
\]
Choosing $\eps=\eps_0(n)$ dimensionally small we obtain \eqref{eq:goal}, since the assumption $\sing(u)=\emptyset$ ensures that $\mathfrak m(B_{\rho})<+\infty$ for all $\rho<1$. 
\end{proof}

\subsection{Proof of \texorpdfstring{\Cref{prop:supf-2D}}{}}
Before explaining the proof we need two preliminary lemmas. The first was already observed in \Cref{rmk:lapfp} and is valid for all $n \geq 2$.
\begin{lemma}\label{lem:W2p}
    If $-\lap u=f(u)$ with $B_1\subset\R^n$, $f\ge0$, $f''\ge0$, and
    \[
    1\le p<+\infty\quad\text{if }n=2,
    \qquad
    1\le p<\frac{n}{n-2}\quad\text{if }n\ge3,
    \]
    then
    \[
    \|\lap u\|_{L^p(B_{1/2})}\le C(n,{p})\big( \|\lap u\|_{L^1(B_1)} +\|f'(u)f(u)\|_{L^1(B_1)} \big).
    \]
\end{lemma}
\begin{proof}
The function $h:=f(u)=-\lap u$ satisfies 
    \[
    h\ge0\quad\text{ and }\quad \Delta h\ge -f'(u)f(u)\quad \text{ in } B_1.
    \]
	    Let $\mathcal G$ be the negative fundamental solution for the Laplacian in $\R^n$ (so $\Delta \mathcal G=\delta_0$). We set
    \[ 
    H(x):=-\mathcal G*(f'(u)f(u)\chi_{B_{3/4}}),
    \]
    so that  $(h-H)_+$ is subharmonic and $H>0$ in $B_{3/4}$.
   Thus, since $\mathcal G \in L^p(B_1)$ by the restriction on \(p\), and
    $|h|=h\le H+(h-H)_+$, we get
    \begin{align*}
    \|h\|_{L^p(B_{1/2})}&\le \|H\|_{L^p(B_{1/2})} + \|(h-H)_+\|_{L^p(B_{1/2})}\\
    &\le \|\mathcal G\|_{L^p(B_{1})}\|f'(u)f(u)\|_{L^1(B_{3/4})}+ C_p \|(h-H)_+\|_{L^1(B_{3/4})}\\
    &\le C_p \, \|f'(u)f(u)\|_{L^1(B_{3/4})}+ C_p \,\|h\|_{L^1(B_{3/4})},
    \end{align*}
  as desired.
\end{proof}




We will also need the following general formula.

\begin{lemma}\label{lem:ODEformula}
    Assume $(u,f)\in\sol(B_1)$, with $B_1\subset\R^n$, and
    $f'(u)|\nabla u|\in L^1_\loc(B_1)$. Define
    \[
    \hat w(r):=\fint_{B_r} f(u), 
    \]
    Then, in the variable $t:=\log(1/r)$, $\hat w$ satisfies  \begin{equation}\label{eq:ODEW}
        \hat w_{tt}-n\, \hat w_t=n\fint_{\partial B_r} f'(u)\,(x\cdot \nabla u).
    \end{equation} 
\end{lemma}

\begin{proof}
Since \((u,f)\in\sol(B_1)\), we have \(u\in H^1_\loc(B_1)\subset W^{1,1}_\loc(B_1)\). Hence, by the chain rule for convex functions and the assumption \(f'(u)|\nabla u|\in L^1_\loc(B_1)\), we have \(f(u)\in W^{1,1}_\loc(B_1)\) and\footnote{Indeed, along a.e. line in a direction \(e\), the function \(s\mapsto u(x+se)\) is absolutely continuous and convexity gives
\[
f(u(x+he))-f(u(x))
=\int_0^h f'(u(x+se))\,\partial_e u(x+se)\,ds
\]
whenever the segment is compactly contained in \(B_1\). The assumed integrability of \(f'(u)|\nabla u|\) then gives the \(W^{1,1}\) chain rule.}
\[
\nabla(f(u))=f'(u)\nabla u\qquad\text{a.e.}
\]
Differentiating $r^{n}\hat w(r)$ once we obtain
\[
\partial_r\big(r^{n}\hat w(r)\big)
= \frac{1}{|B_1|}\de_{r}\int_{ B_r} f(u)=\frac{1}{|B_1|}\int_{\partial B_r} f(u)
=n r^{n-1}\fint_{\partial B_1} f(u(r\,\cdot)).
\]
Differentiating again and scaling back gives
\[
\partial_r\big(r^{1-n}\partial_r\big(r^{n}\hat w(r)\big)\big)
= \frac{n}r\fint_{\partial B_r} f'(u)\,(x\cdot \nabla u).
\]
Noticing that $-r\,\partial_r=\partial_t$ (since $t=\log(1/r)$) and that $\partial_t(r^n)=-nr^n$, we can rewrite this as
\[
n\fint_{\partial B_r} f'(u)\,(x\cdot \nabla u)=\partial_t\big(r^{-n}\partial_t\big(r^n \hat w\big)\big)
= \partial_t\big(\partial_t\hat w-n\hat w\big),
\]
and \eqref{eq:ODEW} follows.
\end{proof}

Next, in the two-dimensional case, we prove a quasi-monotonicity statement for averages of $f(u)$ centered at an arbitrary point.
\begin{lemma}\label{lem:quasiavg2D}
Let $(u,f)\in\sol(B_1)$, with $B_1\subset\R^2$.
There exist universal constants $r_0\in(0,1)$ and $C>0$ such that, setting
\[
A_0:=\int_{B_{9r_0}} f(u)f'(u)\,dx,
\qquad
W(r):=\fint_{B_r} f(u)\,dx+A_0,
\]
one has
\[
\sup_{0<r<r_0} W(r)\le C\, W(r_0).
\]
\end{lemma}

\begin{proof}
Fix \(p_0=3\) and \(r_0>0\) universally small to be chosen later, with \(9r_0<1/4\). 
For \(0<r<r_0\), let \(p_r\) solve
\[
-\Delta p_r=f(u)\quad\text{in }B_r,\qquad p_r=0\quad\text{on }\partial B_r.
\]
By global \(W^{2,p_0}\) estimate and Sobolev embedding,
\[
\|\nabla p_r\|_{L^\infty(B_r)}
\le C r^{1-2/p_0}\|f(u)\|_{L^{p_0}(B_r)}.
\]
Also, applying \Cref{lem:W2p} to the rescaled equation
\[
-\Delta u(er\,\cdot)=(er)^2 f(u(er\,\cdot)),
\]
and using \(B_{er}\subset B_{9r_0}\), we get
\[
\|f(u)\|_{L^{p_0}(B_r)}
\le C r^{2/p_0}\left(\fint_{B_{er}} f(u)\,dx+\int_{B_{er}}f(u)f'(u)\,dx\right)
\le C r^{2/p_0}W(er),
\]
therefore
\begin{equation}\label{eq:poisson-grad-W}
\|\nabla p_r\|_{L^\infty(B_r)}
\le C r W(er)
\qquad\forall\,0<r<r_0.
\end{equation}
Let \(t=\log(1/r)\), and write \(\mathcal W(t):=W(e^{-t})\). Define
\begin{equation}
\label{eq:delay}
\lambda(t):=\frac{\mathcal W_t(t)}{\mathcal W(t)}, \qquad \text{so that}\quad \mathcal W(t-1)=\mathcal W(t)\exp\!\left\{-\int_{t-1}^t\lambda(s)\,ds\right\}.
\end{equation}
All the following differential identities are understood for a.e. radius, or equivalently for a.e. \(t\).
Since \(f(u)\in L^{p_0}_\loc\), we have that \(u\in W^{2,p_0}_\loc\) and therefore \(|\nabla u|\) is locally bounded. Since stability gives \(f'(u)\in L^1_\loc\), \Cref{lem:ODEformula} yields the ODE
\[
\mathcal W_{tt}-2\mathcal W_t
=
2\fint_{\partial B_r} f'(u)\,(x\cdot\nabla u).
\]
Define
\[
A(t):=
-\frac{2\fint_{\partial B_r} f'(u)\,(x\cdot\nabla u)}{W(er)}
=
-\frac{1}{\pi r\,W(er)}
\int_{\partial B_r} f'(u)\,(x\cdot\nabla u),
\]
so that, using \eqref{eq:delay},
we get the delayed Riccati-type ODE
\begin{equation}\label{eq:delayed-W}
\lambda_t+\lambda(\lambda-2)
+A(t)\exp\!\left\{-\int_{t-1}^{t}\lambda(s)\,ds\right\}=0.
\end{equation}
We now record three estimates. First, because
\[
r^2W(r)=\frac1\pi\int_{B_r}f(u)\,dx+A_0r^2
\]
is increasing, 
\[
2-\lambda(t)=2+r\frac{W_r(r)}{W(r)}
=\frac{\partial_r(r^2W(r))}{rW(r)}\ge0,
\]
hence
\begin{equation}\label{eq:lambda-upper}
\lambda(t)\le2.
\end{equation}
Second, we estimate \(A_+(t):=\max\{A(t),0\}\). Set \(h_r:=u-p_r\), where \(p_r\) is as above. Then \(h_r\) is harmonic in \(B_r\) and \(h_r=u\) on \(\partial B_r\). Hence, since $f$ is convex, for a.e. \(r\) we have
\[
r\int_{\partial B_r} f'(u)\partial_\nu h_r
=r\int_{\partial B_r} f'(h_r)\partial_\nu h_r
=r\int_{B_r}\Delta(f(h_r))\ge0.
\]
Here the equality on \(\partial B_r\) uses the trace identity \(h_r=u\), and the last inequality follows by approximating \(f\) with smooth convex functions (equivalently, \(f(h_r)\) is subharmonic because \(h_r\) is harmonic).
Note now that the maximum principle applied to
\(-\Delta p_r=f(u)\ge0\), with \(p_r=0\) on \(\partial B_r\), gives \(p_r\ge0\) in \(B_r\), hence \(\partial_\nu p_r\le0\) on \(\partial B_r\). Thus, using \eqref{eq:poisson-grad-W},
\[
A_+(t)
\le \frac{1}{\pi W(er)}
\int_{\partial B_r} f'(u)(-\partial_\nu p_r)
\le C r\int_{\partial B_r} f'(u),
\]
which implies
\[
\int_{\log(1/r_0)}^\infty A_+(t)\,dt
\le C\int_{B_{r_0}} f'(u)\,dx.
\]
We also observe that, testing stability with a logarithmic cutoff, 
\begin{equation}
\label{eq:stable-log}
    \int_{B_{r_0}} f'(u)\,dx\le \frac{C}{\log(1/r_0)}.
\end{equation}
In particular, choosing \(r_0\) universally small, we may assume
\begin{equation}\label{eq:a-small}
\int_{\log(1/r_0)}^\infty A_+(t)\,dt\le \frac16.
\end{equation}
Third, using again that \(f(u)\in L^{p_0}_{\loc}\), H\"older inequality yields
\[
\int_{B_r}f(u)\,dx\le C r^{2(1-1/p_0)}.
\]
Now, if \(\lambda(t)\ge\beta\) for all large \(t\), then \(W(r)\ge c r^{-\beta}\) for all small \(r\), and therefore
\[
r^2W(r)\ge c r^{2-\beta}.
\]
This contradicts the previous estimate whenever \(\beta>2/p_0\). Thus
\begin{equation}\label{eq:liminf-lambda}
\liminf_{t\to\infty}\lambda(t)\le \frac{2}{p_0}<1.
\end{equation}
We now show how all this implies the desired quasi-monotonicity.

Set
\[
t_0:=\log(1/r_0),\qquad
\Lambda(t):=\int_{t_0}^t\lambda(s)\,ds,
\qquad
M_\Lambda(T):=\max_{t\in[t_0,T]}\Lambda(t).
\]
Since \(\lambda\le2\), we have \(M_\Lambda(t_0+1)\le2\). Also, at every differentiability point of \(M_\Lambda\),
\[
M_\Lambda'(t)>0
\quad\Longrightarrow\quad
\Lambda(t)=M_\Lambda(t),\qquad M_\Lambda'(t)=\lambda(t)>0.
\]
On the open set where \(M_\Lambda'>0\), and for \(t\ge t_0+1\), we have
\[
\Lambda(t)\ge \Lambda(t-1),
\]
hence the exponential factor in \eqref{eq:delayed-W} is at most \(1\).

We claim that
\[
\lambda(t)\le \frac32
\qquad\text{whenever }t\ge t_0+1\text{ and }M_\Lambda'(t)>0.
\]
Indeed, if \(\lambda(t_1)>3/2\) at such a point, then integrating \eqref{eq:delayed-W} forward inside the same increasing interval for \(M_\Lambda\), using \(\lambda\le2\) and the exponential bound, shows that \(\lambda\) cannot drop below
\[
\frac32-\int_{t_1}^\infty A_+(s)\,ds \ge \frac43.
\]
Thus the increasing interval cannot end, and \(\lambda(t)\ge4/3\) for all large \(t\), contradicting \eqref{eq:liminf-lambda}.

Now let $I=(a,b)$ be any maximal interval contained inside $(t_0+1,\infty)$ on which $M_\Lambda'>0$. Since \(0<\lambda\le3/2\), it follows that \(\lambda(2-\lambda)\ge \frac12\lambda\). Thus, integrating \eqref{eq:delayed-W} on $I$ gives
\[
\frac12\int_a^b\lambda(t)\,dt
\le
\lambda(b)+\int_a^b A_+(t)\,dt.
\]
If \(b<+\infty\) then, by maximality and continuity, \(\lambda(b)=0\), and we get
\[
\int_I\lambda(t)\,dt
\le 2\int_I A_+(t)\,dt.
\]
If instead $b=+\infty$, we simply use that $\lambda \leq 2$ to get
\[
\int_I\lambda(t)\,dt
\le 4+ 2\int_I A_+(t)\,dt.
\]
Hence, summing over the disjoint intervals inside $\{M_\Lambda'>0\}\cap (t_0+1,\infty)$ and using \eqref{eq:a-small}, we obtain
\[
\sup_{T\ge t_0+1}M_\Lambda(T)\le M_\Lambda(t_0+1)+4+2\sum_{I}\int_I A_+(t)\,dt \leq M_\Lambda(t_0+1)+4+ \frac13,
\]
therefore
\[
\sup_{t\ge t_0}\int_{t_0}^t\lambda(s)\,ds\le C.
\]
In particular, for every \(0<r<r_0\),
\[
W(r)
=
W(r_0)\exp\!\left\{\int_{t_0}^{\log(1/r)}\lambda(s)\,ds\right\}
\le C\, W(r_0),
\]
which proves the lemma.
\end{proof}

We are finally ready to prove \Cref{prop:supf-2D}. 

\begin{proof}[Proof of \Cref{prop:supf-2D}]
Let \(R:=1/4\) and \(\rho_0:=Rr_0\). Applying \Cref{lem:quasiavg2D} to the rescaled solution \(u(x_0+R\,\cdot)\), with nonlinearity \(R^2f\), and then returning to the original variables, we obtain the following estimate for every \(x_0\in B_{3/4}\): with
\[
A_{x_0}:=\int_{B_{9\rho_0}(x_0)} f(u)f'(u)\,dx,
\qquad
W_{x_0}(r):=\fint_{B_r(x_0)} f(u)\,dx+A_{x_0},
\]
one has
\[
\sup_{0<r<\rho_0} W_{x_0}(r)\le C\,W_{x_0}(\rho_0).
\]
In particular, letting \(r\downarrow0\) we get
\[
f(u(x_0))
\le C\left(\fint_{B_{\rho_0}(x_0)} f(u)\,dx+\int_{B_{9\rho_0}(x_0)} f(u)f'(u)\,dx\right).
\]
Since \(18\rho_0<1/4\), the second term is controlled by \Cref{prop:ff'} after rescaling in \(B_{18\rho_0}(x_0)\):
\[
\int_{B_{9\rho_0}(x_0)} f(u)f'(u)\,dx
\le C\,\|u\|_{L^\infty(B_1)}.
\]
Finally, testing the equation against a cutoff equal to one in \(B_{\rho_0}(x_0)\) and supported in \(B_{2\rho_0}(x_0)\) gives
\[
\fint_{B_{\rho_0}(x_0)} f(u)\,dx\le C\|u\|_{L^\infty(B_1)}.
\]
This proves that \(f(u)\le C\|u\|_{L^\infty(B_1)}\) in \(B_{3/4}\), as desired.
\end{proof}

\subsection{Proof of \texorpdfstring{\Cref{thm:C11-2D}}{}}

\begin{proof}
We first prove the asserted alternative for the singular set. 

We begin by observing that, if \(T_f=+\infty\), then \(\sing(u)=\emptyset\) by \Cref{cor:2Dmems}, since in dimension two the conclusion \(\mathcal H^0(\sing(u))=0\) forces the set to be empty (alternatively, this follows by \cite{Nedev00,Cabre10,actadim9}). Hence, if \(\sing(u)\neq\emptyset\), then necessarily \(T_f<+\infty\).

Let \(z\in\sing(u)\) (in particular, $u(z)=T_f$)
and choose \(s>0\) with \(B_s(z)\Subset B_1\). Then, applying (a rescaled version of) \Cref{prop:supf-2D} inside \(B_s(z)\) gives 
$$
f(T_f)=f(u(z))\le Cs^{-2}\|u\|_{L^\infty(B_s(z))}<+\infty,
$$
as desired.

We now prove the uniform interior $C^{1,1}$ estimates.
Set \(L:=\|u\|_{L^\infty(B_1)}\) and 
$M:=\operatorname*{ess\,sup}_{B_{3/4}}u.$
Since \(f\) is increasing, it follows from \Cref{prop:supf-2D} that
\begin{equation}\label{eq:fM-bound}
f(M)\le C L.
\end{equation}
Set
\[
w:=M-u,\qquad g(s):=f(M-s).
\]
Then \(w\ge0\) in \(B_{3/4}\) and
\[
\Delta w=g(w),
\]
where $g$ is nonnegative, decreasing, and convex, with
\[
g(0)=f(M)\le C L.
\]
If $g(0)=0$, then $g(w)\equiv0$ in \(B_{3/4}\), and the desired estimate follows from the interior estimates for harmonic functions. Hence we assume $g(0)>0$.
By convexity, for every $s>0$,
\begin{equation}\label{eq:gprime-bound}
0\le -s g'(s)\le g(0)-g(s)\le g(0),
\end{equation}
where $g'$ may be understood as either the left or right derivative.

We prove the Hessian bound at points \(y\in B_{1/2}\cap\{w>0\}\). There are two cases.

\smallskip
\noindent\emph{Case 1: \(w(y)\le g(0)/64\).}
Define the intrinsic scale
\[
\rho^2:=\frac{w(y)}{g(0)},\qquad
v(x):=\frac{w(y+\rho x)}{w(y)}.
\]
Since \(y\in B_{1/2}\) and \(\rho\le1/8\), we have \(B_\rho(y)\subset B_{3/4}\), and the rescaled function satisfies
\[
v(0)=1,\qquad v\ge0,\qquad
\Delta v=\frac{g(w(y)v)}{g(0)}.
\]
In particular, $0\le\Delta v\le1$. By Harnack inequality and interior estimates for Poisson equations with bounded right-hand side, this implies
\[
\|v\|_{L^\infty(B_{1/2})}+\|\nabla v\|_{L^\infty(B_{1/2})}\le C.
\]
In particular, after reducing the radius, we obtain
\begin{equation}\label{eq:v-lower-half}
v\ge \frac12\qquad\text{in }B_{\bar r}
\end{equation}
for some universal ${\bar r}>0$.

We now differentiate the equation: for every direction
$e\in \mathbb S^1$,
\[
\Delta \partial_e v
=\rho^2 g'(w(y)v)\,\partial_e v.
\]
Using \eqref{eq:gprime-bound}, \eqref{eq:v-lower-half}, and $\rho^2=w(y)/g(0)$, we get
\[
\big|\rho^2 g'(w(y)v)\big|
\le
\frac{w(y)}{g(0)}\,\frac{g(0)}{w(y)v}
\le \frac{1}{v}\leq 2
\qquad\text{in }B_{\bar r}.
\]
Since $|\partial_e v|\le C$, it follows that
\[
|\Delta\partial_e v|\le C
\qquad\text{in }B_{\bar r},
\]
so interior estimates applied to $\partial_e v$ yield
\[
|D^2v(0)|\le C.
\]
Scaling back, this proves that
\[
|D^2w(y)|
\le C\,\frac{w(y)}{\rho^2}
= C g(0)
= C f(M)
\le C L.
\]

\smallskip
\noindent\emph{Case 2: \(w(y)>g(0)/64\).}
Choose a universal radius \(r_*>0\) so small that \(B_{2r_*}(y)\subset B_{3/4}\) for every \(y\in B_{1/2}\). Since \(0\le\Delta w\le g(0)\), the Harnack inequality with bounded right-hand side gives
\[
w(y)\le C\left(\inf_{B_{r_*}(y)}w+g(0)r_*^2\right).
\]
Taking \(r_*\) smaller if necessary and using \(w(y)>g(0)/64\), we get
\[
w\ge c\,g(0)\qquad\text{in }B_{r_*}(y).
\]
Hence \(|g'(w)|\le C\) in \(B_{r_*}(y)\), by \eqref{eq:gprime-bound}. Also \(0\le g(w)\le g(0)\le CL\) and \(\|w\|_{L^\infty(B_{3/4})}\le 2L\), so standard interior gradient estimates give
\[
|\nabla w|\le C L\qquad\text{in }B_{r_*/2}(y).
\]
Differentiating the equation,
\[
\Delta \partial_e w=g'(w)\partial_e w,
\]
and applying interior estimates to \(\partial_e w\), we obtain \(|D^2w(y)|\le C L\).

The two cases show \(|D^2w|\le CL\) in \(B_{1/2}\cap\{w>0\}\).
Since $w \in W^{2,p}(B_{1/2})$ (since its Laplacian is bounded), its Hessian vanishes a.e. in the set $B_{1/2}\cap\{w=0\}$.
	This proves that $|D^2w|\le CL$ a.e. in \(B_{1/2}\),
	hence
	\[
	\|D^2u\|_{L^\infty(B_{1/2})}\le C\|u\|_{L^\infty(B_1)}.
	\]
	By scaling, we can apply this estimate inside any ball contained inside $B_1$:
	\begin{equation}\label{eq:local-Linfty-Hessian}
	\|D^2u\|_{L^\infty(B_{R/2}(x_0))}
	\le C R^{-2}\|u\|_{L^\infty(B_R(x_0))}\qquad \text{for all balls }B_R(x_0)\subset B_1.
	\end{equation}
	We now combine this bound with interpolation to replace the \(L^\infty\) norm by an \(L^1\) norm. For a ball \(B\subset B_1\), set
	\[
	\mathfrak h(B):=\|D^2u\|_{L^\infty(B)}.
	\]
	By the local estimate just obtained, \(\mathfrak h\) is finite on compactly contained balls, and it is subadditive under finite covers. Let \(B_R(x_0)\subset B_1\). Applying \eqref{eq:local-Linfty-Hessian} to \(B_{R/2}(x_0)\) and using the standard interpolation inequality in \(B_R(x_0)\), for every \(\eta>0\) we get
	\[
	\begin{aligned}
	\mathfrak h(B_{R/4}(x_0))
	&\le C R^{-2}\|u\|_{L^\infty(B_{R/2}(x_0))}\\
	&\le C_\eta R^{-4}\|u\|_{L^1(B_R(x_0))}
	   +C\eta\,\mathfrak h(B_R(x_0)).
	\end{aligned}
	\]
	Multiplying by \(R^4\) and using \(\|u\|_{L^1(B_R(x_0))}\le \|u\|_{L^1(B_1)}\), we obtain
	\[
	R^4\mathfrak h(B_{R/4}(x_0))
	\le C\eta R^4\mathfrak h(B_R(x_0))
	   +C_\eta\|u\|_{L^1(B_1)}.
	\]
	Choosing \(\eta>0\) so small that \(C\eta\le\varepsilon_0\), where \(\varepsilon_0\) is the constant in \Cref{lem:covering} with \(\beta=4\), and then applying \Cref{lem:covering}, we conclude that
	\[
	\|D^2u\|_{L^\infty(B_{1/2})}
	\le C\|u\|_{L^1(B_1)}.
	\]
	\end{proof}

\begin{remark}\label{rmk:weiss-monneau}
Let $(u,f)\in\sol(B_1)$ with $B_1\subset\R^2$, and
suppose   \(z\in\sing(u)\). Then, by \Cref{thm:C11-2D}, \(T_f<+\infty\) and \(f(T_f)<+\infty\). Define
\[
w:=T_f-u,\qquad Z:=\{w=0\},\qquad g(s):=f(T_f-s),\qquad \lambda:=g(0)=f(T_f).
\]
At points of \(Z\), the equation
\[
\Delta w=g(w),
\]
together with the nonnegativity of $w$ and the $C^{1,1}$ estimate suggests an analogy with singular points in the classical obstacle problem. Indeed, if we define
\[
w_{z,r}(x):=\frac{w(z+rx)}{r^2},
\]
then
\[
\Delta w_{z,r}=g(r^2w_{z,r})\to g(0)
\qquad\text{as }r\downarrow0.
\]
Thus one may hope to adapt the Weiss and Monneau monotonicity formulas for the obstacle problem
\cite{Weiss99,Monneau03} (see also \cite{FigalliJEDPObstacle} for an overview), to prove that the Hessian of $w$ exists at $z$, and then perhaps deduce that $w \in C^2$.
In the present setting, however, the error terms in these monotonicity formulas are governed by
\[
g(0)-g(w).
\]
The \(C^{1,1}\) estimate above gives a local quadratic bound \(w\le A r^2\) around \(z\), for some finite constant \(A\). Hence one expects Weiss-type almost-monotonicity, and then uniqueness of quadratic blow-ups via a Monneau-type formula, provided a Dini condition of the form
\[
\int_0^1 \omega(r)\,\frac{dr}{r}<+\infty,
\qquad
\omega(r):=g(0)-g(A r^2),
\]
holds. In particular, without such an additional assumption, only \(C^{1,1}\) regularity is expected to hold.
\end{remark}

\appendix
\section{Proof of some Lemmas and Examples}

\subsection{Details on the proof of \texorpdfstring{\eqref{eq:specification}}{}}\label{app:ACTA}
We briefly justify the pointwise convergence \eqref{eq:specification} used in \Cref{thm:compactness}.
In general, if $(v,g)\in\sol(B_1)$ and $v$ is not constant, then
\[
\big|\{v=\sup_{B_1}v\}\big|=0.
\]
Indeed:
\begin{itemize}
\item If $\sup_{B_1}v=T_g$, then $g'(v)\in L^1_\loc$ forces $|\{v=T_g\}|=0$.
\item If $\sup_{B_1}v<T_g$, then $v\in C^2$ and if $|\{v=\sup v\}|>0$, the equation would give $g(\sup v)=-\Delta v=0$ (since $\Delta v=0$ a.e. in the region $\{v=\sup v\}$). Therefore, $v$ would be harmonic and hence constant (since it attains an interior maximum), contradiction.
\end{itemize}
Now fix $m<\sup_{\Omega}u$ and pick a point $x$ such that $u(x)<m$.
Up to a subsequence, we can assume that $u_k\to u$ a.e. Thus, 
since $u_k(x)\to u(x)$ and $f_k\to f$ locally uniformly on $(-\infty,m]$, we have
$f_k(u_k(x))\to f(u(x))$ for a.e.\ such $x$. Letting $m\uparrow\sup_\Omega u$ yields \eqref{eq:specification}.

\subsection{Proof of \texorpdfstring{\Cref{lem:gmtintro}}{}}\label{app:3}
\begin{proof}
If \(\alpha q\ge n\), then for every \(x\in B_1\),
\[
r^{\alpha-n}\int_{B_r(x)}\varphi
\le |B_r|^{1-1/q}r^{\alpha-n}\|\varphi\|_{L^q(B_r(x))}
\le C r^{\alpha-n/q}\|\varphi\|_{L^q(B_r(x))}\to0\quad\text{ as }r\downarrow0.
\]
Hence \(E=\emptyset\). We may therefore assume \(\alpha q<n\).

Fix $\delta>0$ and for each $x\in E$ choose $r_x\in(0,\delta)$ such that
\[
\int_{B(x,r_x)}\varphi \ge \tfrac12\,r_x^{n-\alpha}.
\]
By Vitali's covering lemma, there exists a countable subcollection $\{B(x_i,r_i)\}$ of pairwise disjoint balls such that
\[
E \subset \bigcup_{i} B(x_i,5r_i).
\]
Also, by H\"older,
\[
\int_{B(x_i,r_i)}\varphi^q
\ge |B_{r_i}|^{1-q}\left(\int_{B(x_i,r_i)}\varphi\right)^q
\ge c(n)\,r_i^{n-\alpha q}.
\]
Therefore, since the balls  $\{B(x_i,r_i)\}$  are disjoint,
\[
\mathcal{H}^{n-\alpha q}_\delta(E)
\le \sum_i (10r_i)^{n-\alpha q}
\le C(n)\sum_i \int_{B(x_i,r_i)}\varphi^q
\le C(n)\int_{E_\delta}\varphi^q,
\]
where $E_\delta$ is the $\delta$-neighborhood of $E$.
Letting $\delta\downarrow0$ we find $\mathcal{H}^{n-\alpha q}(E)<+\infty$. Since $\alpha>0$, this implies $\mathcal{H}^n(E)=0$. Then using $\varphi\in L^q_\loc$ and $\cap_{\delta>0}E_\delta =E,$ we find that in fact $\mathcal{H}^{n-\alpha q}(E)=0$.
\end{proof}

\subsection{An integral version of the counterexample of Villegas}\label{app:Villegas}
In this appendix we modify the construction in \cite{villegas21} to show that there exists a convex, increasing, superlinear $f\in C^2(\R)$ and a radial singular weak stable solution $u^\star$ such that 
\begin{equation}\label{eq:villegas}
\rho_k^{2-n}\int_{B_{\rho_k}} f'(u^\star(x))\,dx \to 0\quad \text{for a sequence $\rho_k\to0$, even though $0\in\sing(u^\star)$.}
\end{equation}
In fact, as the argument shows, one can replace $\rho_k^{2-n}$ by $\rho_k^{\delta-n}$ for any $\delta>0$.

For consistency, only in this subsection we follow the notation of \cite{villegas21} and denote the dimension by $N$.
Villegas starts from a radial decreasing $\Psi\in C^\infty(\overline B_1\setminus\{0\})$ such that
\[
0<\Psi(r)\le \frac{2(N-2)}{r^2}\qquad \forall\,r\in(0,1],
\]
and constructs the unique radial function $\omega\in H^1(B_1)$ solving the linear PDE
\begin{equation}\label{eq:Ppsi}
\begin{cases}
-\lap \omega(x)= \left(\Psi (x)-\frac{N-1}{|x|^2}\right)\omega (x)& \text{ in } B_1,\\
\omega (x)= 1 & \text{ on } \partial B_1.
\end{cases}\tag{$P_\Psi$}
\end{equation}
Then one defines $(u,f)\in\sol(B_1)$ by imposing
\[
u'(r)=-\omega(r),\quad f'(u(r))=\Psi(r),\quad u(1)=0\quad \text{and}\quad u(0)=\int_0^1 \omega(r)\,dr.
\]

In the following Proposition we construct $\Psi$ so that $\int_0^1 \omega(r)\,dr=+\infty$ and $\Psi(0)=+\infty$ (so $0\in\sing(u)$), while keeping $\Psi$ very small in average along a sequence of scales $y_n$. 

\begin{proposition}\label{prop:villegasmod}
Let $N\ge 10$, $\beta>1$, and $\delta\in(0,\tfrac12]$. Then there exist a radial decreasing $\Psi\in C^\infty(\overline{B_1}\setminus \{0\})$ and a sequence $y_n\downarrow0$ such that:
\begin{enumerate}
\item[(i)] $0<\Psi(r)\le \dfrac{2(N-2)}{r^{2}}$ and $\Psi'(r)<0$ for all $0<r\le1$.
\item[(ii)] $\Psi(r)=r^{-\delta}$ for all $r\in [y_n^\beta,y_n]$.
\item[(iii)] If $\omega\in H^1(B_1)$ solves \eqref{eq:Ppsi}, then $\int_0^1 \omega(r)\,dr=+\infty$.
\end{enumerate}
\end{proposition}

\begin{proof}
We follow \cite{villegas21} closely. Set $x_1=1$ and construct sequences $x_n\downarrow0$ and $y_n\downarrow0$ inductively as follows:
choose $y_n$ so small that
\[
y_n< x_n e^{-1/x_n^3}\quad\text{ and }\quad y_n^{-\delta}>\frac{2(N-2)}{\left(x_n e^{-1/x_n^3}\right)^2},
\]
and then set $x_{n+1}:=y_n^\beta/4$. This ensures
\[
x_{n+1}< y_n^\beta <y_n<x_n e^{-1/x_n^3}<x_n.
\]
Define $\Psi$ on each interval $[x_{n+1},x_n)$ by
\[
\Psi(r) :=
\begin{cases}
2(N-2)\,r^{-2} & \text{on }[x_n e^{-1/x_n^3},x_n),\\
\text{a smooth decreasing transition} & \text{on }[y_n, x_n e^{-1/x_n^3}),\\
r^{-\delta} &\text{on }[ y_n^\beta,y_n),\\
\text{a smooth decreasing transition} &\text{on }[\tfrac12y_n^\beta, y_n^\beta),\\
2(N-2)\,r^{-2} & \text{on }[x_{n+1},\tfrac12y_n^\beta).
\end{cases}
\]
The choice of $y_n$ makes the transitions possible while keeping $\Psi$ decreasing and satisfying $\Psi(r)\le 2(N-2)r^{-2}$.
Property (iii) is proven exactly as in \cite[Proposition 2.3]{villegas21}: it uses the size of $\Psi$ on the intervals $[x_n e^{-1/x_n^3},x_n)$ and the assumption $N\ge 10$.
\end{proof}

Taking $\beta=\dfrac{N-\delta}{N-2}>1$, properties (i) and (ii) in \Cref{prop:villegasmod} imply that
\[
\int_{\{|x|<y_n\}} f'(u)\,dx \le C\int_{\{|x|<y_n^\beta\}}|x|^{-2}\,dx
 + C\int_{\{y_n^\beta<|x|<y_n\}}|x|^{-\delta}\,dx
\le C y_n^{\beta(N-2)} + C y_n^{N-\delta}.
\]
Since $\beta(N-2)=N-\delta$, this gives
\[
\int_{\{|x|<y_n\}} f'(u)\,dx \le C y_n^{N-\delta},
\]
hence
\[
y_n^{2-N}\int_{\{|x|<y_n\}} f'(u)\,dx \le C y_n^{2-\delta}\to0,
\]
which corresponds exactly to \eqref{eq:villegas} with $\rho_n=y_n$.

\subsection{A useful logarithmic comparison lemma}\label{app:logratio}

\begin{lemma}\label{lem:logratio}
Let $f:(t_0,T_f)\to(0,+\infty)$ be convex.

\smallskip
\noindent\textup{(i)}
Assume $T_f<+\infty$, $f(t)\to+\infty$ as $t\uparrow T_f$, and
\[
\int_{t_0}^{T_f} f(s)\,ds = +\infty.
\]
Then
\[
\limsup_{t\uparrow T_f}\frac{\log f(t)}{\log f'(t)} \;\ge\; \frac12.
\]

\smallskip
\noindent\textup{(ii)}
Assume $T_f=+\infty$, $f(t)\to+\infty$ as $t\to+\infty$, and $f'(t)>1$ for all sufficiently large $t$. Then
\[
\limsup_{t\to+\infty}\frac{\log f(t)}{\log f'(t)} \;\ge\; 1.
\]
\end{lemma}

\begin{proof}
\medskip
\noindent\emph{(i)}
Assume by contradiction that
\[
L:=\limsup_{t\uparrow T_f}\frac{\log f(t)}{\log f'(t)} < \frac12.
\]
Choose $\varepsilon\in\bigl(0,\tfrac12-L\bigr)$ so that $L\le \tfrac12-2\varepsilon$.
For $t$ close enough to $T_f$, we have $\log f(t)>0$, $\log f'(t)>0$, and
\[
\frac{\log f(t)}{\log f'(t)}\le \frac12-\varepsilon,
\]
i.e.
\[
f'(t) \ge f(t)^{\,q},
\qquad q=\frac{2}{1-2\varepsilon}>2.
\]
Since $f$ is increasing, set $y=f(t)$. From this inequality we obtain $dt/dy \le y^{-q}$, hence
\[
\int_{t_1}^{T_f} f(s)\,ds
=\int_{f(t_1)}^{\infty} y\,\frac{dt}{dy}\,dy
\le \int_{f(t_1)}^{\infty} y^{1-q}\,dy <+\infty,
\]
contradicting $\int_{t_0}^{T_f} f=+\infty$.

\medskip
\noindent\emph{(ii)}
Assume by contradiction that
\[
L:=\limsup_{t\to+\infty}\frac{\log f(t)}{\log f'(t)} < 1.
\]
Then for some $\varepsilon>0$ and all $t$ large enough,
\[
\frac{\log f(t)}{\log f'(t)}\le 1-\varepsilon,
\qquad\text{so}\qquad
f'(t)\ge f(t)^{\frac{1}{1-\varepsilon}} = f(t)^{1+\delta},\quad \delta=\frac{\varepsilon}{1-\varepsilon}>0.
\]
This differential inequality forces $f$ to blow up in finite time, contradicting $T_f=+\infty$.
\end{proof}

\subsection{Construction of suitable cut-off functions}\label{app:cutoff}
We construct the cutoff lemma used in the proof of \Cref{prop:aprioribound}.

\begin{lemma}\label{lem:logp}
Given $p>1$ and $\eps>0$, there exists $\zeta\in C^2_c(B_1)$ such that $\zeta\ge \chi_{B_{1/2}}$ and
\[
\int_{B_1} w \Big(|\nabla \zeta|^2+|\zeta\lap\zeta| \Big)\, dx
\le \eps\int_{B_1} w \,|\log w |^{2p}\,\zeta^2\,dx +C,
\]
for all nonnegative measurable $w:B_1 \to \R$, where $C=C(n,p,\eps)$ is a large constant.
\end{lemma}

\begin{proof}
Let $q:=\frac{p+1}2\in (1,p)$ and choose a radially decreasing cutoff $0\le \zeta\le1$ with
\[
\chi_{B_{1/2}}\le \zeta \le \chi_{B_1},\qquad \zeta>0\ \text{on }B_1,
\]
and such that near $\partial B_1$ (i.e.\ for $\zeta$ small) it satisfies 
$$
\zeta(|x|)=e^{-(1-|x|)_+^{-\frac{1}{q-1}}},\qquad \text{thus}\quad \partial_r\zeta= -\frac{1}{q-1}\zeta\,|\log\zeta|^{q}.
$$
Note that, with this choice, for $x$ close to $\partial B_1$ we have
\[
\zeta\lap\zeta=\zeta\left(\partial_{rr}\zeta+(n-1)\frac{\partial_r\zeta}{r}\right)
=\frac{1}{(q-1)^2} \zeta^2 |\log\zeta |^{2q}\Big(1-\frac{q}{|\log\zeta|}-\frac{(q-1)(n-1)}{|x|\,|\log\zeta|^q}\Big)\sim \zeta^2|\log\zeta|^{2q}.
\]
Since away from $\partial B_1$ we can bound $|\nabla \zeta|^2+|\zeta\lap\zeta|\leq C\zeta^2$, we conclude that
$$
|\nabla \zeta|^2+|\zeta\lap\zeta|\leq \bar C\,\zeta^2\left(1+|\log\zeta|^{2q}\right)\qquad \text{in }B_1,
$$
and therefore
\[
\int w\Big(|\nabla \zeta|^2+|\zeta\lap\zeta| \Big)
\leq \bar C \int w\,\zeta^2\left(1+|\log\zeta|^{2q}\right).
\]
Now, recalling that $w \geq 0$ and $0\leq \zeta\leq 1$, a standard Young-type splitting yields that, for any $\eps>0$,\footnote{Indeed, set
$t:=\sqrt{1+|\log\zeta|^{2q}}\ge 1$, so that $1+|\log\zeta|^{2q}=t^2$.
If $t\le \varepsilon|\log w|^{p}$, then $w t^{2}\le \varepsilon^{2}w|\log w|^{2p}$, giving the desired inequality.
Otherwise, $|\log w|\le (t/\varepsilon)^{1/p}$, hence $w\le \exp((t/\varepsilon)^{1/p})$, and therefore
\[
w\,t^2 \le t^2 \exp\!\bigl((t/\varepsilon)^{1/p}\bigr).
\]
Since $\zeta^{-2}=e^{2|\log\zeta|}=\exp\!\bigl(2(t^2-1)^{1/(2q)}\bigr)$ and $q<p$, the function
\[
t\mapsto t^2\exp\!\Bigl((t/\varepsilon)^{1/p}-2(t^2-1)^{1/(2q)}\Bigr)
\]
is bounded on $[1,+\infty)$. This implies that $w\,t^2\le C_{\varepsilon,p,q}\,\zeta^{-2}$, and therefore the desired inequality holds also in this case.}
\[
w\,\bigl(1+|\log\zeta|^{2q}\bigr)
\le \varepsilon^2\,w\,|\log w|^{2p} + C_{\varepsilon,p,q}\,\zeta^{-2}.
\]
so multiplying by $\zeta^2$ and integrating yields
\[
\int w\Big(|\nabla \zeta|^2+|\zeta\lap\zeta| \Big)
\leq \bar C\eps^2\int w|\log w|^{2p}\zeta^2 + C_{\eps,p,q}.
\]
Assuming without loss of generality that $\eps$ is sufficiently small so that $\bar C\eps\leq 1$, the result follows.
\end{proof}
\medskip

\noindent
{\it Acknowledgments:} The first author is grateful to the Marvin V. and Beverly J. Mielke Fund for supporting his stay at IAS
Princeton, where part of this work was done. The second author was supported by the European Research Council (ERC) under
Grant Agreement No. 948029 while he was at ETH, and by the Giorgio and Elena Petronio Fellowship while he was at the IAS.

\medskip

\noindent
\textit{AI statement:} The mathematical work in this paper was carried out by the authors without the use of AI.

\printbibliography
\end{document}